\documentclass[a4paper,12pt]{article}
\usepackage[labelfont=bf]{caption}
\usepackage{tikz}
\usepackage{float}
\usepackage{diagbox}
\usepackage[round,comma]{natbib}
\usepackage{geometry}
\usepackage{amsmath,amsfonts,amssymb,amsthm,booktabs,color,graphicx,hyperref,url}

\theoremstyle{plain}
\newtheorem{theorem}{Theorem}[section]

\newtheorem{rem}{Remark}[section]

\newtheorem{lemma}{Lemma}
\newtheorem{corollary}{Corollary}

\theoremstyle{definition}

\hypersetup{%
breaklinks=true,
colorlinks=true,
citecolor=blue,
linkcolor=red,
urlcolor=magenta,
}

\newcommand{\R}{\mathbb{R}}
\newcommand{\argmax}{\mathop{\mathrm{argmax}}}

\newcommand{\esp}{\mathbb{E}}

\newcommand{\prob}{\mathbb{P}}
\newcommand{\1}{\mathbf{1}}

\newcommand{\dd}{D}

\newcommand{\hausdorff}{\mathcal{H}}
\newcommand{\diam}{\Delta}

\newcommand{\In}{\interval{n}}
\newcommand{\cX}{\mathcal{X}}

\newcommand{\constA}{\mathfrak{a}}
\newcommand{\constB}{\mathfrak{b}}

\newcommand{\interval}[1]{{\{1,\dotsc,#1\}}}
\newcommand{\intervalb}[2]{{\{#1,\dotsc,#2\}}}

\newcommand{\gbar}{\begin{tikzpicture}[
	anchor=base, baseline,
	outer sep=0pt, 
	]
	\node[](A) {$\gamma$};
	\draw[shorten >=2pt, shorten <=3pt] ([yshift=-0.5pt]A.west) to ([yshift=-0.5pt]A.east);
	\end{tikzpicture}}

  \title{Statistical analysis of a hierarchical clustering algorithm with outliers}
  \author{%
  Nicolas Klutchnikoff%
  \thanks{%
  Univ Rennes, CNRS, IRMAR (Institut de Recherche Mathématique
  de Rennes) - UMR 6625, F-35000 Rennes, France
  }
  \and
  Audrey Poterie%
  \thanks{%
    Univ Bretagne Sud, CNRS, LMBA (Laboratoire de Mathématiques Bretagne Atlantique) - UMR 6205, F-56000 Vannes, France
  }
  \and
  Laurent Rouvière%
  \thanks{%
  Univ Rennes, CNRS, IRMAR (Institut de Recherche Mathématique
  de Rennes) - UMR 6625, F-35000 Rennes, France
  }}
  \date{}
  
\begin{document}
  \maketitle
  
  \hrule

  \begin{abstract}
    \footnotesize{\noindent
      It is well known that the classical single linkage algorithm usually fails to identify clusters in the presence of outliers. In this paper, we propose a new version of this algorithm, and we study its mathematical  performances. In particular, we establish an oracle type inequality which ensures that our procedure allows to recover the clusters with large probability under minimal assumptions on the distribution of the outliers. We deduce from this inequality the consistency and some rates of convergence of our algorithm for various situations. Performances of our approach is also assessed through simulation studies and a comparison with classical clustering algorithms on simulated data is also presented.
    
      \medskip

    \noindent
    {\bf Keywords:} Clustering, Outliers contamination, Single linkage.\\
    {\bf AMS Subject Classification:} 62G20, 62H30}
  \end{abstract}

  \hrule

\section{Introduction}

In unsupervised learning, clustering refers to a very broad set of tools which aim at finding a partition of the data into dissimilar groups so that the observations within each group are quite similar to each other. Considered as one of the most important questions in unsupervised learning, there is a vast literature on this paradigm. We refer the reader to~\cite{har75,jain1988algorithms,duda2012pattern} and references therein for a broad overview.
Moreover, many clustering methods have been developed and studied, such as the $k$-means algorithm~\citep{macqueen1967some}, the hierarchical clustering methods~\citep{johnson1967hierarchical}, the spectral clustering algorithms~\citep{ng2002spectral}, the model-based clustering approaches~\citep{mclachlan1988mixture} or density based methods as DBSCAN~\citep{ester1996density}. Clustering plays an important role in explanatory data analysis and has been used in many fields including pattern recognition~\citep{satish2006kernel}, image analysis~\citep{filipovych2011semi}, document retrieval, bioinformatics~\citep{yamanishi2004protein,zeng2012clustering} and data compression~\citep{gersho2012vector}. Overall, clustering tools are often used to help users understand the data structure. Furthermore, with the massive increase in the amount of collected and stored data, clustering methods can also be used as dimensionality reduction techniques~\citep{yengo:hal-00764927}.

In this paper, we consider a mathematical framework close to the one used in ~\cite{mahevo09,arias2011spectral,auklro15}. The data are generated according to a mixture of several distributions whose supports are assumed to be disjoint in order to identify the groups. Furthermore, we assume that the data are contaminated by outliers, that is observations that do not belong to any of the supports. Many authors have studied theoretical properties of nearest neighbor graphs, hierarchical and spectral clustering algorithms in a similar context. For instance,~\cite{mahevo09} provides an in-depth analysis of $k$-nearest neighbor graphs while~\cite{arias2011clustering} studies the performance of spectral clustering algorithms and single linkage algorithms under assumptions on both the distances between supports and the number of outliers.

Single linkage algorithm is a hierarchical method which consists of recursively merging the two closest clusters in terms of minimal distance. Although this procedure has many interesting properties, it is well known that its performance is much lower  in the presence of outliers. This issue comes from two different phenomena. On the one hand, the procedure may wrongly detect small clusters among the outliers. On the other hand, during the recursive clustering procedure, a chain of outliers may lead to merging two groups which contain observations that belong to two different supports. To overcome these problems, we propose an automatic procedure based on a new analysis of the dendrogram produced by the hierarchical agglomerative clustering in terms of minimal distance. This new procedure can be viewed as a simple modification of the classical single linkage algorithm adapted to the presence of outliers. Moreover, the proposed method allows the detection of clusters with low-dimensional geometrical structures. This last property, shared with spectral clustering, is of primary interest in several modern applications, see~\cite{arias2011clustering} for more details. Like spectral clustering, our data-driven procedure only requires knowledge of the number of groups to identify the clusters with high probability (under mild assumptions on the size of the clusters). Furthermore, our approach offers a decisive advantage over spectral clustering in terms of time complexity.

The paper is organized as follows. In Section~\ref{sec:framework}, we introduce the mathematical framework, the model assumptions, and we define a criterion, called {\it clustering risk}, to measure performance of clustering algorithms. Section~\ref{sec:aggloclust} describes the single linkage clustering algorithm and shows, through simple examples, that this algorithm often fails to recover true clusters in the presence of outliers. We then build a new variant of this algorithm which takes into account the possible presence of outliers in the data. This new procedure does not require, besides the number of groups desired, the calibration of any additional parameter. In section~\ref{sec:main_results}, we prove the efficiency of our procedure by exhibiting an oracle-type inequality. Some consistency results and rates of convergence are then deduced from this inequality.  Section~\ref{sec:simus} is devoted to compare the performance of our method with other classical clustering algorithms through several synthetic datasets.  The proofs are gathered at the end of the paper, in Section~\ref{sec:proofs}. The proposed clustering method has been implemented in \textsf{R} and the source code is available at~\url{https://github.com/klutchnikoff/outliersSL}.

\section{Mathematical framework}\label{sec:framework}

In this section, we define a general probabilistic model to generate data which locally belong to low-dimensional structures $S_{\!1},\dotsc,S_{\!M}$ and which possibly contain outliers collected in a set $S_{0}$. We also specify what we expect from a clustering procedure in this framework, and we define a risk to quantify the performance of such a procedure.

\subsection{Generative model}\label{sec:model}

We specify in this section how the data are generated in $S_{\!1},\dotsc,S_{\!M}$ and how the outliers are sampled outside these structures.
We are given $n$ independent $[0,1]^{\dd}$-valued random variables $X_{1},\dotsc,X_{n}$, and we assume that their common distribution $\prob$  can be written as a mixture of $M+1$ distributions $\prob_{0},\dotsc,\prob_{M}$.
More precisely, for $0\leq \varepsilon<1$ and  a vector of convex weights $(\gamma_1,\dotsc,\gamma_M)$, $\prob$ can be decomposed as follows:
\begin{equation}\label{eq:def_P}
  \prob = \varepsilon \prob_{0} + (1-\varepsilon) \sum_{i=1}^{M}
  \gamma_{i} \prob_{i}.
\end{equation}
The value $\varepsilon$ represents the proportion of outliers contained in the data while $\prob_{0}$ stands for the distribution of these outliers.
The second term of the right-hand side of this equation is, up to the factor $1-\varepsilon$, the distribution of the \emph{actual data} or \emph{non-outlier data}. These actual data are distributed into $M$ disjoint groups: $\gamma_i$ represents
the weight of the $i$-th group and $\prob_i$ denotes its distribution. For any $i\in\intervalb{1}{M}$, let $S_{i} = \operatorname{supp}(\prob_i)$
be the compact set of all points $x\in[0,1]^\dd$ for which any neighborhood $A$ of $x$ satisfies $\prob_i(A)>0$. We also define the set
\begin{equation*}
  S_{0} = [0,1]^\dd \setminus \left(\bigcup_{i=1}^M S_{i}\right),
\end{equation*}
and we assume that $\operatorname{supp}(\prob_0)\subseteq \operatorname{adh}(S_{0})$ and $\prob_0(\operatorname{adh}(S_{0})\cap S_{i})=0$ for $i\in\interval{M}$. Here $\operatorname{adh}(A)$ stands for the adherence of a set $A$. These assumptions on $\prob_0$ mean in particular that outliers are defined as observations that do not belong to the supports of the $M$ groups.

\subsection{Assumptions}
\label{sec:assumptions}
Clusters are usually identified by high-density regions separated by low-density regions. For instance,~\cite{har75} defines clusters
as connected components of the level sets of the density of the observations. Moreover, the geometry of a cluster often corresponds to low-dimensional structures such as submanifolds of $[0,1]^D$ \citep[see for example][]{arias2011clustering,arias2011spectral}.
We consider a similar framework and detail the assumptions of our model below.

\begin{description}
  \item[(A1)] For each $i\in\intervalb1{M}$,  the set $S_{i}$ is connected. Moreover,
        \[
          \delta=\min_{1\leq i<j\leq M}
          \min\{\|x-y\|:
          (x,y)\in S_{i}\times S_{\!j}
          \} > 0,
        \]
        where $\|\cdot\|$ stands for the Euclidean norm.
\end{description}

Assumption \textbf{(A1)} ensures that the  supports are disjoint and well separated. This implies that the model is identifiable since the decomposition~\eqref{eq:def_P} of $\prob$ is then unique, up to any permutation of the indexes $\{1,\dotsc,M\}$. Note also that, under this assumption, the dataset $\mathbb{X}_n = \{X_1,\dotsc, X_n\}$ is split into $M+1$ well-defined \emph{groups}: $\mathbb{X}_n\cap S_{0}$ corresponds to the outliers whereas for $i\geq 1$, $\mathbb{X}_n\cap S_{i}$ correspond to the \emph{true clusters} we want to recover.

Throughout the paper, for any $0\leq s \leq \dd$, we denote by $\hausdorff^s$ the $s$-dimensional Hausdorff outer measure. We recall that, if $s$ is an integer, this measure agrees with ordinary ``$s$--dimensional surface area'' on regular sets. In particular, $\hausdorff^\dd$ is the standard Lebesgue measure on the ambient space $\R^\dd$. We refer the reader to~\citet{evans2015measure} for more details on this topic. We also define:
\begin{equation*}
  s_{i}=\operatorname{dim}_{H}(S_{i})
  \quad\text{and}\quad
  d = \max\{s_i : i\in\interval{M}\},
\end{equation*}
where $\operatorname{dim}_{H}(S_{i})$ denotes the Hausdorff dimension of the set $S_{i}$, that is the unique real number $s\in[0,\dd]$ such that $\hausdorff^t(S_{i}) = \infty$ if $t<s$ and $\hausdorff^t(S_{i}) = 0$ if $t>s$. Notice that if $S_{i}$ is a submanifold of $\R^\dd$, its Hausdorff dimension $s_i$ corresponds to its classical dimension.

\begin{description}
  \item[(A2)] There exists $\kappa_0 >0$ such that, for any $A\subseteq S_{0}$, we have $\prob_{0}(A)\le \kappa_0 \hausdorff^{\dd}(A)$.
\end{description}

This assumption relates to the distribution of the outliers and can be reformulated as follows: $\prob_{0}$ is absolutely continuous with respect to $\hausdorff^{\dd}$, with bounded Radon-Nikodym derivative.
This implies, in some sense, that the outliers  are nowhere dense in $S_{0}$ and thus prevents having clusters that correspond to groups of outliers.

\begin{description}
  \item[(A3)]  For any $i\in\interval{M}$, there exists $\kappa_i >0$ such that, for any $A\subseteq S_{i}$, we have $\prob_{i}(A)\ge \kappa_i^{-1} \hausdorff^{s_{i}}(A)$.
\end{description}

Assumption~\textbf{(A3)} relates to the distribution of \emph{actual data} and ensures that each $\prob_{i}$ is quite dense on $S_{i}$. Note in particular that $\prob_{i}$ can be singular with respect to  $\hausdorff^{s_i}$ ($\prob_i$ is singular with respect to $\hausdorff^\dd$ as soon as $s_i<\dd$).

Assumptions~\textbf{(A1)},~\textbf{(A2)} and~\textbf{(A3)} are classical in the clustering setting. The first one guarantees identifiability of the model while the other two  highlight differences between outliers and actual data:  the former are diffused while the latter are densely distributed into their supports.

Geometric assumptions on the supports $S_{i}$ are also needed. To state them, we denote by $B(x,r)$  the Euclidean ball centered at $x\in\R^\dd$ with radius $r>0$ and by $\Gamma$ the usual gamma function.
Recall that, for any $s>0$, the function $\eta(s)=\pi^{s/2}\Gamma^{-1}(1+s/2)$ generalizes to non-integer parameters the volume of the unit ball in dimension $s$.

\begin{description}
  \item[(A4)] There exists $\kappa_c\geq 1$  such that, for any $i\in\interval{M}$, $x\in S_{i}$ and $0<r\le \Delta_{i} = \operatorname{diam}(S_{i})$,
        \begin{equation*}
          \label{eq:regularity-Si}
          \kappa_c^{-1}\le\frac{\hausdorff^{s_{i}}(S_{i}\cap B(x,
          r))}{\eta(s_{i}) r^{s_{i}}}\le\kappa_c.
        \end{equation*}
        Here $\operatorname{diam}(S_{i})=\max\{\|x-y\|:x\in S_{i},y\in S_{i}\}$ denotes the diameter of the support $S_{i}$.
\end{description}

Assumption \textbf{(A4)} prevents the sets $S_{i}$ from being ``too narrow'' in some places. A similar assumption is made in~\citet{arias2011clustering}. Note also that, if $S_{i}$ is a submanifold that satisfies a \emph{reach} condition, then \textbf{(A4)} is automatically fulfilled~\citep[see][and references therein]{MR2320821}.

\begin{description}
  \item[(A5)] For any $i\in\intervalb{1}{M}$, the Hausdorff dimension $s_i$ of $S_{i}$ agrees with its Minkowski-Bouligand dimension, that is:
        \begin{equation*}
          s_i = \lim_{r\to0} \frac{\log(N_r(S_{i}))}{\log(1/r)},
        \end{equation*}
        where $N_r(S_{i})$ denotes the minimal number of open balls of radius $r$ required to cover $S_{i}$.
\end{description}

Assumption~\textbf{(A5)} is necessary to obtain  sharp bounds on the covering numbers $N_r(S_{i}),r>0$ for any $i\in\intervalb{1}{M}$. Indeed, in general, we only have
\begin{equation*}
  s_i
  \leq \liminf_{r\to0} \frac{\log(N_r(S_{i}))}{\log(1/r)}
  \leq \limsup_{r\to0} \frac{\log(N_r(S_{i}))}{\log(1/r)}
  \leq \dd.
\end{equation*}
Here we assume that the limit inferior matches with the limit superior and that these limits equal $s_i$. This technical assumption is not too restrictive.
We offer two simple generic examples. First, if $S_{i}$ is a submanifold of $\R^\dd$ then \textbf{(A5)} holds. Indeed, in this case, the Hausdorff dimension and the Minkowski-Bouligand dimension both match with the usual dimension of $S_{i}$. Next \textbf{(A5)}  is also satisfied if $S_{i}$ is a self-similar set. Indeed, using Assumption~\textbf{(A4)} with $r=\diam_i$, we obtain that
$0<\hausdorff^{s_{i}}(S_{i})<+\infty$. This implies that $S_{i}$ satisfies the \emph{open set condition} which allows us to conclude that both the Hausdorff and the Minkowski-Boulingand dimensions match with the affinity dimension of the self-similar set $S_{i}$~\citep[see][chapter~9]{MR3236784}.

\begin{description}
  \item[(A6)] Let $\gamma_*=\min\{\gamma_i : i\in\interval{M}\}$, $\gamma^*=\max\{\gamma_i : i\in\interval{M}\}$ and $\gbar=\gamma_*-\gamma^*/2$. We assume that:
        \begin{equation*}
          \gamma^*< 2\gamma_*
          \quad\text{and}\quad
          0\leq\varepsilon<\gbar/(1+\gbar).
        \end{equation*}
\end{description}
This assumption allows differentiating actual clusters from the set of outliers.
It implies that the sizes of the actual clusters should be of the same order since the largest cluster cannot be twice as large as the smallest one. The number of allowed outliers is constrained by
the difference in size of the groups. The more homogeneous the sizes of the groups are, the higher the proportion of outliers can be. For instance, $\varepsilon$  must be equal to 0 when the largest cluster is twice as large as the smallest one, {\it i.e.} when $\gamma^*= 2\gamma_*$. This proportion could increase when the gap between cluster sizes reduces. In particular, it could reach  $1/(2M+1)$ when $\gamma^*=\gamma_*=1/M$.

\subsection{Clustering risk}

We aim at finding a clustering procedure that groups together the data that lie within the same set $S_{i}$, for each $i\in\interval{M}$.  Regarding the outliers, they can be assigned to any other group or collected into a specific group by the procedure. A clustering procedure consists of splitting the data $\mathbb{X}_n$ into $M$ disjoint clusters. In other words, a clustering algorithm provides a family of clusters $\cX = \{\cX_1,\dotsc,\cX_M\}$ such that, for any $1\leq i\neq j\leq M$,
\begin{equation}
  \cX_i\neq\emptyset,
  \qquad
  \cX_i\cap\cX_j=\emptyset,
  \qquad\text{and}\qquad
  \bigcup_{i=1}^M \cX_i\subseteq \mathbb{X}_n.
\end{equation}

Observe that the family $\cX$ may not cover the whole set $\mathbb{X}_n$. It could be the case if the algorithm reveals some outliers
that are not assigned to any cluster. In our context, a clustering procedure is efficient if each cluster contains {all the} observations from (only) one of the supports $S_{i}$, $i\in\interval{M}$. It means that
there exists a unique {permutation} $\pi\in\Pi_m$ {from the} set of all permutations of $\interval{M}$, such that, for any $i\in\interval{M}$, the data $\mathbb{X}_n \cap S_{i}$ are included into $\cX_{\pi(i)}$.
In this context, we measure the performance of a clustering procedure by the {\it clustering risk} defined as
\begin{equation}
  \label{eq:defclustrisk}
  \mathcal R_{n}(\cX) =  \prob(\forall\pi\in\Pi_M, \exists i\in\interval{M}, \; \mathbb{X}_n \cap S_{i} \nsubseteq \cX_{\pi(i)}),
\end{equation}
where $\cX = \{\cX_{1},\dotsc,\cX_{M}\}$ is the clustering family selected by the clustering procedure. This quantity is the probability that the clustering procedure does not correctly recover one subset
of observations from at least one of the $S_{i}$'s. The smaller the risk, the better the clustering procedure.

\section{Single linkage algorithm for outliers}
\label{sec:aggloclust}
Many clustering algorithms have been studied in a context similar to our setting. For instance,~\cite{mahevo09},~\cite{arias2011clustering} and~\cite{arias2011spectral}
prove that algorithms based on pairwise distances ($k$-nearest neighbor graph, spectral clustering...) are efficient
as soon as the supports $S_{i},i=1,\dotsc,M$ are sufficiently separated. The single linkage hierarchical clustering algorithm has
also been investigated by~\cite{arias2011clustering} and~\cite{auklro15}. However, it is well known that this algorithm is sensitive to outliers.
We propose here to address the weaknesses of this algorithm in the presence of outliers.

\subsection{Agglomerative clustering with single linkage}

Many hierarchical clustering algorithms rely on the notion of $r$--connected set of points in $[0,1]^\dd$, where $r$ is a nonnegative real number. A subset $A\subseteq \R[0,1]^\dd$ is said to be $r$--connected, if
\begin{equation*}
  B(A,r/2)= \bigcup_{a\in A} B(a,r/2)
\end{equation*}
is a connected set, from a topological point of view. In particular $A=\{x,y\}$ is $r$--connected if $\|x-y\|\leq r$.
The single linkage algorithm may be defined with this notion of connected set of points.
For any $r\geq0$, the set $B(\mathbb{X}_n,r/2)$ can be expanded into $M(r)\in\{1,\dotsc,n\}$ connected components denoted by $B_m(r)$ for $m\in\interval{M(r)}$. These connected components provide a partition of $\mathbb{X}_n$ into $M(r)$ clusters defined, for any $m\in\interval{M(r)}$, by
\begin{equation*}
  \mathcal{Y}_m(r) =  B_m(r) \cap \mathbb{X}_n.
\end{equation*}
The family  $\mathcal{Y}(r)= \lbrace \mathcal{Y}_m(r) \colon m\in\interval{M(r)}\rbrace$ provides clusters of the single linkage algorithm with radius $r$.

We can observe that the number of possible families $\mathcal{Y}(r)$ is finite when we let $r$ move in $\R^+$.
Indeed, as $r$ increases the clustering process consists of recursively merging the clusters. To see that, consider the single linkage distance between  two $r$-connected components $\mathcal{Y}_m(r)$ and  $\mathcal{Y}_{m'}(r)$.
It  is defined as the distance between the two closest members between these components
\begin{equation*}
  \operatorname{dist}(\mathcal{Y}_m(r),\mathcal{Y}_{m'}(r) )= \inf \lbrace \|X_k- X_l\| \colon  X_k \in \mathcal{Y}_m(r), \, X_l \in \mathcal{Y}_{m'}(r)\rbrace.
\end{equation*}
At the beginning, for $r=\rho_0=0$ we have a first family
\begin{equation*}
  \mathcal{Y}(\rho_0) = \lbrace \mathcal{Y}_m(\rho_0), \, m\in\interval{M(\rho_0)}\rbrace.
\end{equation*}
Observe that if $X_i\neq X_j$ for all $1\leq i\neq j\leq n$ then $M(\rho_0)=n$ and each cluster $\mathcal Y_m(\rho_0)$
corresponds with one observation. Next the two closest clusters are merged according to the (smallest) distance $\operatorname{dist}(\cdot)$.
Denote by $\rho_1>0$ the distance between the two closest clusters in $\mathcal Y(\rho_0)$, we obtain the second family
\begin{equation*}
  \mathcal{Y}(\rho_1)= \{ \mathcal{Y}_m(\rho_1), m\in\interval{M(\rho_1)}\}.
\end{equation*}
This process is then recursively repeated until all (distinct) observations belong to a single cluster.
We denote by $K$ the (random) number of iterations, observe that $K\leq n-1$ almost surely.

\begin{rem}
  Let us make some general remarks about this procedure.
  At every step $k$ with $1\leq k \leq K$, the new selected radius $\rho_k$ is larger than the previous one: $\rho_k > \rho_{k-1}$.
  This radius corresponds to the distance between the two closest clusters belonging to $\mathcal{Y}(\rho_{k-1})$. Moreover, for any $\rho \in [\rho_k,\rho_{k+1}[$ with $0\leq k \leq K-1$ and $\rho_K=\infty$, we have
  \begin{equation*}
    \mathcal{Y}(\rho) = \mathcal{Y}(\rho_k).
  \end{equation*}
\end{rem}

At the end of the process, we obtain a sequence $\mathcal Y(\rho_0),\dotsc,\mathcal Y(\rho_{K-1})$ of partitions of the data. The aim is to determine how to choose one partition in this sequence. In other words, we have to select a radius in the sequence $\rho_0,\dotsc,\rho_{K-1}$.
Since the number of clusters is known, a natural way is to choose the radius such that
the associated number of clusters is close to $M$. More precisely, it is usually chosen such that
\begin{equation*}
  \widehat\rho_{n,SL} \in  \argmax_{\rho\in\{\rho_k \colon k\in\intervalb0{K-1}\}} \{M(\rho) \geq M\}.
\end{equation*}
Observe that $ \widehat\rho_{n,SL}$ exists as soon as each support $S_{i}$ contains at least one observation.
This algorithm is known to be consistent without outliers (i.e. if $\varepsilon=0$) and under assumptions close to ours~\citep{arias2011clustering,auklro15}.

\subsection{Dealing with outliers}\label{sec:SL}

In the presence of outliers, clusters in $\mathcal Y(\widehat\rho_{n,SL})$ may fail to recover supports $S_{i}$ with high probability.
We provide two toy examples to show that.

\paragraph{Example 1}
Figure~\ref{fig:contre_ex_SL}
displays clusters obtained by the classical single linkage algorithm on 2 datasets. The first one (\textsf{data1}) contains two groups and these
two groups are perfectly identified by the algorithm. For the second one (\textsf{data2}), one outlier has been added. We observe
that this single outlier defines one group while all the other observations are in the second group. Here, the performance of the classical single linkage is dramatically affected by this outlier.

\begin{figure}[H]
  \centering
  \includegraphics{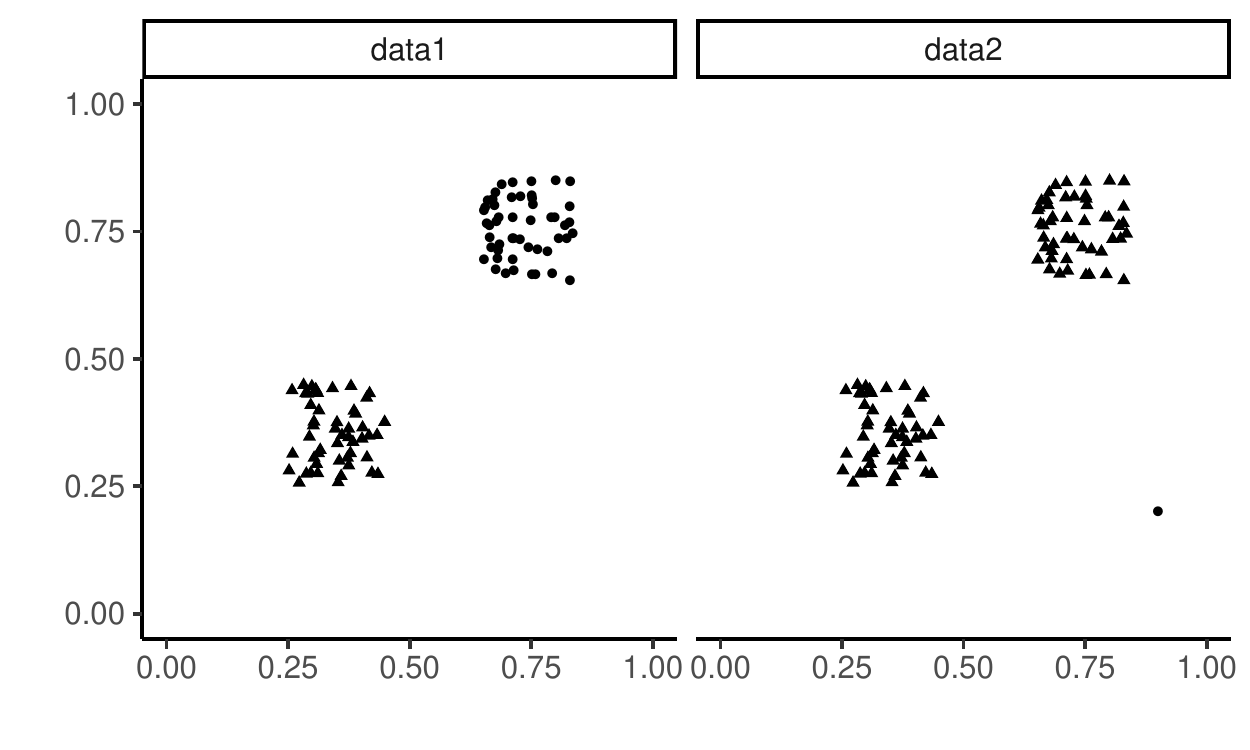}
  \caption{Results of classical single linkage algorithm performed on 2 datasets.}
  \label{fig:contre_ex_SL}
\end{figure}

\paragraph{Example 2} We consider data generated according to the following univariate distribution :
\begin{equation*}
  \prob = \frac{1-\varepsilon}2(\prob_1+\prob_2) + \varepsilon \prob_0,
\end{equation*}
where $\prob_1=\delta_{-1}$, $\prob_2=\delta_1$, $\prob_0=\mathcal U([-3,3])$ and $\varepsilon>0$. For $\varepsilon$ small enough, it is easily seen that
assumptions presented in section~\ref{sec:assumptions} are satisfied. However, simple calculations show that the single linkage procedure fails. Indeed, for any $n >2$,
\begin{equation*}
  \mathcal R_{n}(\mathcal Y(\widehat\rho_{n,SL}))\geq \frac23-\frac83 \frac1{(n+1)\varepsilon}.
\end{equation*}
As $n$ increases, the clustering risk tends to 2/3.

\subsection{OSL algorithm}

Results in the previous section show that the single linkage procedure is generally not efficient in the presence of outliers. To address this issue,~\cite{arias2011clustering} considers a modified version of the procedure that requires the
knowledge of the minimal separation distance $\delta$  defined in Assumption~\textbf{(A1)}. Moreover, to prove some consistency results, it is also assumed that the minimal distance between the outliers and the actual data is bounded below by $\delta$. Here, we adopt a different strategy that, from our point of view, seems both more realistic and reasonnable: we let the user choose the number of groups rather than the parameter $\delta$.
With this in mind, we propose to take the cardinality of the $M$-{th} largest clusters into account in our procedure. More precisely, our method consists of selecting the radius of the single linkage algorithm which maximizes the size of the $M$-{th} largest cluster.
In the following, we describe the procedure.

Recall that for any fixed radius $r>0$, the agglomerative clustering, presented  at the beginning of Section~\ref{sec:aggloclust}, provides $M(r)$ clusters
\begin{equation*}
  \mathcal Y(r)=\{\mathcal Y_m(r),m\in\interval{M(r)}\}.
\end{equation*}
With no loss of generality, we reorder the indices of the $r$-connected components in $\mathcal Y(r)$ such that
\begin{equation*}
  |\mathcal Y_1(r)| > |\mathcal Y_2(r)| > \dotsc> |\mathcal Y_{M(r)}(r)|.
\end{equation*}
In the event of a tie, {\it i.e.} if $|\mathcal Y_m(r)| = |\mathcal Y_{m'}(r)|$, several rules can be applied to break them. We use the following convention: $\mathcal Y_m(r)$ is declared ``bigger'' than $\mathcal Y_{m'}(r)$ if
\begin{equation*}
  \min\{k\in\{1,\dotsc,n\}:X_k\in \mathcal Y_m(r)\}<\min\{k\in\{1,\dotsc,n\}:X_k\in \mathcal Y_{m'}(r)\}.
\end{equation*}
This means that tie breaking is done by the smallest index in the cluster.

Our robust single linkage clustering procedure proposes to consider only the $M$ (which is assumed to be known) largest clusters and to merge the other clusters together. Formally, for a fixed value of $r>0$ we consider the $M$ clusters
\begin{equation}
  \label{eq:defmathcX}
  \cX_{1}(r)=\mathcal Y_{1}(r),\dotsc,\cX_{M}(r)=\mathcal Y_{M}(r).
\end{equation}
If $M(r)<M$, we define $\cX_m(r)=\emptyset$, for any $m\in\intervalb{M(r)+1}{M}$. Otherwise, when $M(r)> M$, the observations that belong to
\begin{equation*}
  \cX_{0}(r)= \bigcup_{m=M+1}^{M(r)}\mathcal Y_{m}(r)
\end{equation*}
are not assigned to any group $\cX_m(r)$, $m\in\interval{M}$. For a fixed value of $r>0$, this procedure provides the family of clusters $\cX(r)=\{\cX_1(r),\dotsc,\cX_M(r)\}$.

To propose a data-dependent choice of the radius $r$, a few remarks are necessary. Too small values of $r$  may result in large values of $M(r)$.
In this case, the $M$ largest clusters could be too small and the clustering procedure may fail to recover all the supports $S_{\!1},\dotsc,S_{\!M}$.
On the opposite, too large values of $r$ may increase both the risk to gather observations from different supports in the same cluster, and the risk to obtain clusters defined by outliers.
The algorithm's performance then depends greatly on the choice of the radius $r$. With this in mind, we select the radius in $\{\rho_k \colon k\in\intervalb0{K-1}\}$ which maximizes the size of the $M$-th cluster:
\begin{equation}
  \label{eq:choixharr}
  \widehat r_n = \max \argmax_{\rho\in\{\rho_k \colon k\in\intervalb0{K-1}\}} |\cX_M(\rho)|.
\end{equation}

\begin{rem}
  The main difference compared to the single linkage clustering is that this algorithm selects the partition which maximizes the size of the $M$-{th} cluster. Observations that belong to $\cX_{0}(\widehat r_n)$ are not assigned to any group and might be considered as outliers.
\end{rem}

\section{Main results}\label{sec:main_results}

The selection procedure~\eqref{eq:choixharr} defines the family of clusters $\cX(\widehat r_n)=\{\cX_m(\widehat r_n),m\in\intervalb1M\}$ whose clustering risk is given by:
\begin{equation*}
  \mathcal R_n(\cX(\widehat r_n))=\prob(\forall\pi\in\Pi_M, \exists i\in\interval{M}, \; \mathbb{X}_n \cap S_{i} \nsubseteq \cX_{\pi(i)}(\widehat r_n)).
\end{equation*}
The following theorem provides an oracle-type inequality which ensures that this clustering risk is close to the optimal clustering risk, {\it i.e.}, the one achieved with the best value of $r$.
\begin{theorem}
  \label{theo:clustriskhatrn}
  Assume {\bf (A1)} and {\bf (A6)} hold. Let $\eta_0>0$ and $\eta_1>0$ be such that
  \begin{equation*}
    \eta_0=1-[(1-\varepsilon)(1+\gbar)]^{-1}
    \quad\text{and}\quad
    \frac{4\eta_1}{1-\eta_1}=\frac{\gamma_*}{\gamma^*}-\frac{1}{2}.
  \end{equation*}
  Then for all $0<\eta\leq\min(\eta_0,\eta_1)$ and all $n\geq M$, the clustering risk for clusters $\cX(\widehat r_n)$ satisfies
  \begin{equation}
    \label{eq:ineg_oracla_hatr}
    \mathcal R_n(\cX(\widehat r_n)) \leq \inf_{r>0}\mathcal R_n(\cX(r))+2M\exp(-\psi(\eta)(1-\varepsilon)\gbar n)
  \end{equation}
  where for $\eta>0$ $\psi(\eta)=(1+\eta)(\log(1+\eta)-1)+1>0$.
\end{theorem}

This theorem ensures that the data-driven selection of $r$ proposed in~\eqref{eq:choixharr} is efficient for $n$ large enough. Indeed, since $\psi(\eta)(1-\varepsilon)\gbar>0$, inequality~\eqref{eq:ineg_oracla_hatr} guarantees that the performance of our procedure is optimal, up to a remainder term which tends to zero at an exponential rate. In particular, if there exists a specific value $r_n$ of $r$ such that the clustering risk of $\cX(r_n)$ tends to $0$ as $n$ increases, Theorem~\ref{theo:clustriskhatrn} implies that the risk of $\cX(\widehat r_n)$ also tends to $0$.

To study the clustering risk of $\cX(r)$ for a given value of $r>0$, we define the parameters
\begin{equation}
  \label{eq:def-constAB}
  \constA=\frac{\gamma_*(\kappa^*\kappa_c)^{-1} \eta_{*}(d)}{1+\gbar}
  \quad\text{and}\quad
  \constB=\eta(\dd)\kappa_0,
\end{equation}
where
\begin{equation}\label{eq:kappa-eta-star}
  \kappa^{*}=\max_{i\in\interval{M}} \kappa_i
  \quad\text{and}\quad
  \eta_*(d)
  =\min_{0\leq s\leq d}\eta(s)
  =\min(1,\eta(d)).
\end{equation}

These parameters measure to some extent the complexity of the problem. Indeed, $\constB$ essentially depends on the density of the outliers through the parameter $\kappa_0$. Problems with sparse
outliers will correspond to a small value of $\constB$. The second parameter $\constA$ is related to the distribution
of the actual data in their supports $S_{i},i\in\{1,\dotsc,M\}$ and on the regularity of these supports. Regular supports ($\kappa_c$ small)
with a large density of observations ($\kappa^*$ small) lead to large values of $\constA$. To summarize, difficult problems
correspond to small $\constA$ and/or large $\constB$.

The following theorem controls the clustering risk of $\cX(r)$ in terms of the parameters of the model.
\begin{theorem}
  \label{thm:cluster-identification}
  Under assumptions~{\bf (A1)}--{\bf (A6)} we have, for any $0< r < \min(\min_i \Delta_i, \delta)$ and
  for any $\eta$ such that $0<\eta<\eta_0$
  \begin{equation}
    \label{eq:risqueclustrn}
    \mathcal R_n(\cX(r))\leq \Lambda r^{-d}\exp(-\constA n r^d)+ n\varepsilon (\constB \varepsilon n r^\dd)^{\lfloor\frac\delta r\rfloor}+2M\exp(-\psi(\eta)(1-\varepsilon)\gbar n),
  \end{equation}
  where $\Lambda$ is a positive constant specified in the proof of the Theorem.
\end{theorem}
The upper bound in~\eqref{eq:risqueclustrn} is governed by the first two terms since the last term generally tends to zero much faster.
Recall that the cluster family $\cX(r)$ may fail to recover the true clusters if one of these two conditions is satisfied:
\begin{enumerate}
  \item Observations in a same support are not $r$-connected: there exists $i\in\intervalb{1}{M}$ such that $\mathbb{X}_n\cap S_{i}$ is not  $r$-connected;
  \item Some observations that belong to different supports are $r$-connected: there is a $r$-connected path between $S_{i}$ and $S_{\!j}$ for $(i,j)\in\intervalb{1}{M}^2$ with $i\neq j$.
\end{enumerate}
The first term on the right-hand side of~\eqref{eq:risqueclustrn} corresponds to the first condition. Unsurprisingly, this term is small for large values of $r$ and/or $\constA$. The second term is related to the second condition and, unlike the first term, it tends to decrease when $r$ decreases. This second term also depends on the distribution of the outliers. We observe that it is equal to zero when there is no outlier,  and it increases as the outlier parameter $\constB$ and/or the proportion of outliers $\varepsilon$ grows.

Observe also that the minimal distance between supports $\delta$
occurs through the exponent $\lfloor\delta/r\rfloor$. If $\constB \varepsilon n r^\dd<1$, the second error term decreases as $\lfloor\delta/r\rfloor$ increases. Moreover, we can remark that the upper bound involves two dimensions: the (maximal) Hausdorff dimension $d$ of the support $S_{i}$ and the dimension $\dd$ of the outlier space $S_{0}$. For fixed values of $\dd$, we could obtain slower rates as $d$ increases because it is more difficult to connect observations for large values of $d$. On the opposite, keeping $d$ constant, rates could be faster when $\dd$ grows because the probability to connect observations in $S_{0}$ decreases. Combining Theorems~\ref{theo:clustriskhatrn} and~\ref{thm:cluster-identification} we obtain:
\begin{theorem}\label{theo:comb-th1-th2}
  Under assumptions~{\bf (A1)}--{\bf (A6)} we have, for all $n\geq M$, for any $0< r < \min(\min_i \Delta_i, \delta)$ and all  $\eta\leq\min(\eta_0,\eta_1)$:
  \begin{equation*}
    R_n(\cX(\widehat r_n)) \leq
    \inf_{r>0} \left\{\Delta r^{-d}\exp(-\constA n r^d)+ n\varepsilon (\constB\varepsilon n r^\dd)^{\lfloor\frac\delta r\rfloor}\right\}
    +4M\exp(-\psi(\eta)(1-\varepsilon)\gbar n).
  \end{equation*}
\end{theorem}

If we intend to prove any consistency results regarding $\mathcal R_n(\cX(\widehat r_n))$, we have to exhibit at least one value of $r$ such that the first terms in this upper bound tends to zero.
The following corollary provides sufficient conditions for the consistency of the clustering procedure, {\it i.e.}, conditions for which we have
\begin{equation}
  \label{eq:defconsistency}
  \lim_{n\to +\infty}\mathcal R_n(\cX(\widehat r_n))=0.
\end{equation}
Except for $D$ and $d$, all parameters ($\delta$, $\kappa^*$, $\varepsilon$...) may vary with $n$. For simplicity, we only let $\varepsilon$ vary with $n$
and keep all other parameters fixed in the conditions.

\begin{corollary}
  Under the assumptions of Theorem~\ref{theo:comb-th1-th2}, consistency~\eqref{eq:defconsistency} holds if  either $d<\dd$ or $D=d$ and $\varepsilon < (\constB\log n)^{-1}$.
\end{corollary}

We obtain this result by taking $r^d = D\log(n)/(\constA d n)$.
This corollary ensures that consistency holds as soon as the Hausdorff dimensions of the supports $S_{\!1},\dotsc,S_{\!M}$ are smaller than the dimension $\dd$ of the ambient space. When these dimensions match, the proportion of outliers should tend to $0$ much faster than $1/\log n$. Observe also that without outliers ($\varepsilon=0$), convergence occurs for all $d\leq\dd$.
Using similar tools, we can obtain many rates of convergence with respect to the proportion $\varepsilon$ of outliers. Some examples are gathered in the following corollary.

\begin{corollary}
  Under the assumptions of Theorem~\ref{theo:comb-th1-th2}, there exists two universal constants $C_1$ and $C_2$ such that the following propositions hold:
  \begin{enumerate}
    \item Few outliers:  if $\varepsilon=\exp(-n)$ then:
          \begin{equation*}
            R_n(\cX(\widehat r_n)) \leq C_1 n\exp(-C_2n).
          \end{equation*}

    \item Small dimensions of the supports: if $D>d+1$ then
          \begin{equation*}
            R_n(\cX(\widehat r_n)) \leq C_1n\exp(-C_2n^{1/(d+1)}).
          \end{equation*}
    \item Large dimensions of the supports with few outliers: if $d\leq D\leq d+1$ and $\varepsilon=n^{-\beta}$ with $\beta\geq 1-D/(d+1)$, then
          \begin{equation*}
            R_n(\cX(\widehat r_n)) \leq C_1n^{d/(d+1)}\exp(-C_2n^{1/(d+1)}).
          \end{equation*}
    \item Large dimensions of the supports with many outliers: if $d\leq D\leq d+1$ and $\varepsilon=n^{-\beta}$ with  $0<\beta<1-D/(d+1)$, then
          \begin{equation*}
            R_n(\cX(\widehat r_n)) \leq C_1n^{(1-\beta)d/D}\exp(-C_2n^{1+(\beta-1)d/D}).
          \end{equation*}
  \end{enumerate}
\end{corollary}

To summarize, in each of the above situations, we obtain an upper bound of the form
\[
  R_n(\cX(\widehat r_n)) \leq C_1 n^A\exp\left(-C_2 n^{B}\right)
\]
where $A$, $B$ are given positive constants that depend on the complexity of the problem.
We would like to highlight some key points. The fastest rates of convergence are reached in case~1, when the proportion of outliers is at its lowest level. In case~2, the data lie into sets whose dimension is much smaller than the dimension of the ambient space and the rates of convergence only depend on the parameter $d$. In the last two cases, $d$ is close to $\dd$. Rates of convergence mainly depend on the proportion of outliers.

\section{Numerical experiments}\label{sec:simus}

This section is dedicated to the evaluation of our clustering procedure through two different simulation studies that are described below.
First, we simulate data according to the design introduced in Section~\ref{sec:model}. This part can be viewed as an illustration of Theorem~\ref{theo:comb-th1-th2}. More precisely, we examine and compare the performance of both our procedure (OSL) and the single linkage algorithm (SL) in the presence of outliers. We also consider the spectral clustering algorithm \citep[SC, see][for a brief presentation of this method]{von2007tutorial}.  Indeed, SC shares several properties with our method such as the ability to detect clusters with low-dimensional geometrical structures and the fact that the number of groups is the main tuning parameter of the algorithm.
Next, we use several labelled clustering problems that can be found in the literature. Note that the corresponding data do not necessarily follow our model or satisfy the assumptions introduced in Section~\ref{sec:model}. For these datasets, OSL is compared with SL, SC and also other common clustering algorithms such as the $k$-means algorithm \citep[KMeans, see][]{macqueen1967some}, the trimmed $k$-means (TKMeans), an extension of KMeans introduced in~\citet{cuesta1997trimmed}, the density-based spatial clustering of applications with noise algorithm \citep[DBSCAN, see][]{ester1996density} and its hierarchical version \citep[HDBSCAN, see][]{campello2013density}.
In all experiments, the number of groups $M$ is assumed to be known. Simulation studies have been performed in \texttt{R}. A \texttt{R} implementation of OSL is available at \url{https://github.com/klutchnikoff/outliersSL}.

\subsection{Sensitivity to the parameters of the model}
\label{sec:simu1}

Here, we examine the performance of OSL in the mathematical framework described in Section~\ref{sec:model}. As stated in Section~\ref{sec:main_results}, the efficiency of the proposed clustering algorithm depends on the complexity of the clustering problem, measured through the parameters $(\constA,\constB, \delta, \epsilon)$. Among them, the most sensitive are the intergroup distance $\delta$ and the proportion of outliers $\varepsilon$. Of course, the larger $\varepsilon$ and the smaller $\delta$, the more difficult the problem. The following subsection describes the considered scenarios.

\subsubsection{Description of the models}

To simulate our data, we use three different models with different values for the intergroup distance $\delta$, the proportion of outliers $\varepsilon$ and the sample size $n$. More precisely, for each model, we consider two values of $\delta$ (one ``small value'' which corresponds to an \emph{easy} case and one ``large value'' for a \emph{tricky} case), five values of $\varepsilon$ (equally spaced between $0$ and $0.2$ with a step of 0.05), and two different sample sizes ($n=200$ and $n=500$).
For each model, both groups and outliers are uniformly sampled over their supports $S_{i},i=1,\dotsc,M$ and $S_{0}$. We now describe the three models.

\paragraph{Squares model} Data are grouped in three distinct squares with similar areas. We use the same weights for each group. Easy and tricky cases correspond to intergroup distances 0.35 and 0.07 respectively, see Figure~\ref{fig:perfs_square}.

\paragraph{Concentric circles} This model consists of two nested rings with weight 0.4 for the smallest ring and 0.6 for the largest ring. Intergroup distances are fixed to 2.6 (easy) and 1.6 (tricky), see Figure~\ref{fig:perfs_spiral}. Outliers are generated only between the two rings.

\paragraph{Sine model}  This model includes 3 groups with various shapes. The first group is tight and represents the sine curve while the two others are two compact squares. We use the same weights for each group, they are separated from 1.18 (easy) and 0.76 (tricky), see Figure~\ref{fig:perfs_sinus}.
\bigskip

Simple calculations show that these models satisfy assumptions \textbf{(A1)}-\textbf{(A6)} when $\varepsilon<1/11$ for the ``concentric circles'' model and $\varepsilon<1/7$ for the two others. We can also remark that the Hausdorff dimension for all supports equals 2, except for the sine group where it equals 1.

\subsubsection{Performance according to the proportion of outliers and the sample size}

Through various numerical experiments based on the scenarios described in the previous section, the performance of OSL is evaluated and compared with the one about SL and SC. Regarding the implementation and the calibration of the algorithms, OSL and SL have been implemented by using the functions \texttt{hclust} (package \texttt{fastcluster}) and \texttt{cutree} (package \texttt{stats}). SC has been implemented following~\cite{ng2002spectral} and using the function \texttt{specc} (package \texttt{kernlab}). The scaling parameter is set to the optimal value provided by $\texttt{specc}$, and we consider 20 different random starts for the $k$-means step of the algorithm. The three algorithms require the knowledge of the number of groups $M$ which is assumed to be known. Observe that for our proposed data-driven approach OSL, $M$ is the only parameter that needs to be tuned. In each scenario, the clustering risk  (\ref{eq:defclustrisk}) of each clustering algorithm is approximated by its empirical estimator computed over $B=1000$ Monte Carlo replications
\begin{equation}
  \label{eq:est_risque_MC}
  \frac{1}{B} \sum_{b=1}^{B}\1_{\left\lbrace \forall \, \pi \in \Pi_{M}, \, \exists  i=1,\dotsc,M, \; \mathbb{X}_n^b \cap S_{i} \nsubseteq \cX^b_{\pi(i)}\right\rbrace },
\end{equation}
where $\cX^b = \lbrace \cX^b_1,\dotsc,\cX^b_M \rbrace$ denotes the clusters obtained by the procedure on the $b$-th Monte Carlo sample of data $\mathbb{X}_n^b = \{X^b_{1},\dotsc,X^b_{n}\}$. Estimator~\eqref{eq:est_risque_MC} corresponds to the proportion of Monte Carlo replications in which the clustering procedure does not correctly recover one subset of observations from at least one of the $S_{i}$'s.
Figures~\ref{fig:perfs_square}-\ref{fig:perfs_sinus} display the three models and the empirical estimate~\eqref{eq:est_risque_MC} of the clustering risk according to $\varepsilon$, $n$ and $\delta$ for all algorithms and each model.

\begin{figure}[hp]
  {\small
    \centering
    {\footnotesize
      \begin{tabular}{c}
        \renewcommand{\arraystretch}{2}
        \begin{tabular}{ccc}
          (a)                                                                     & \hspace{2cm} & (b)                                                                   \\
          \includegraphics[width=5cm,height=4cm]{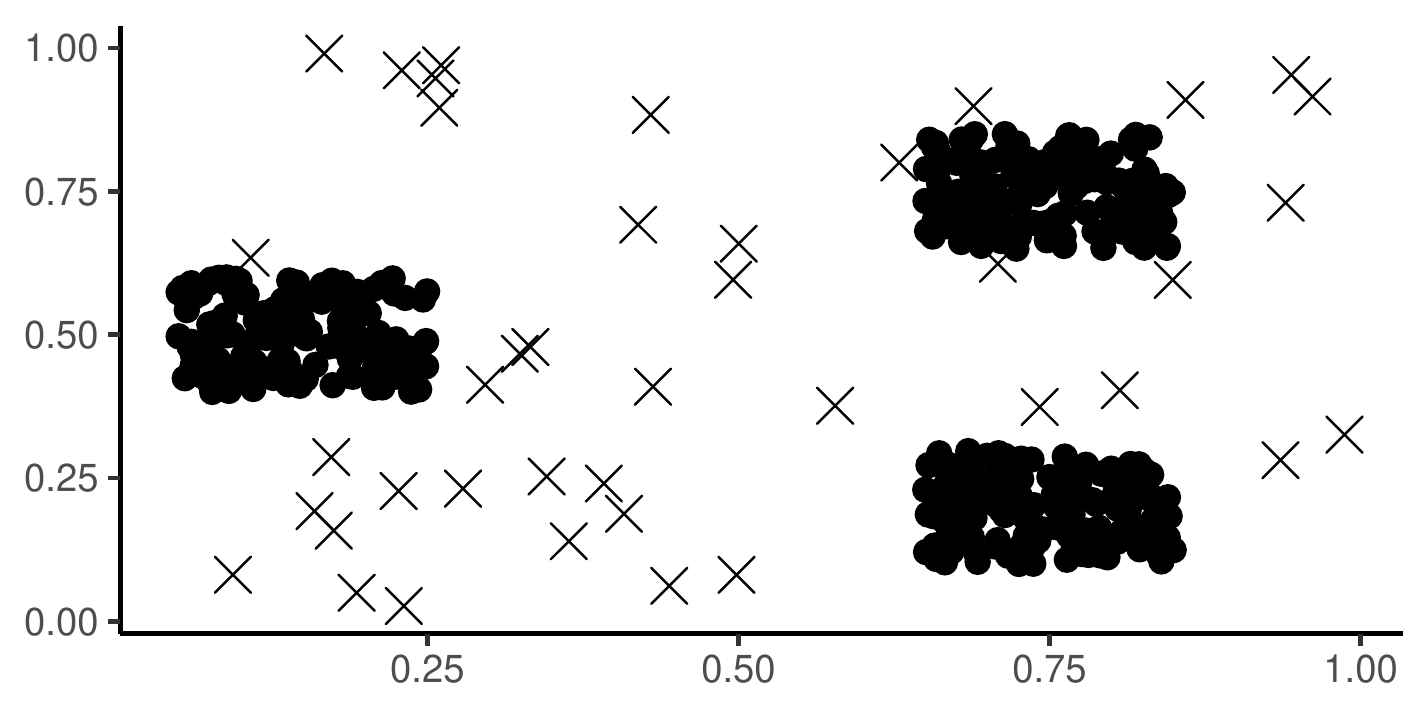} & \hspace{2cm} & \includegraphics[width=5cm,height=4cm]{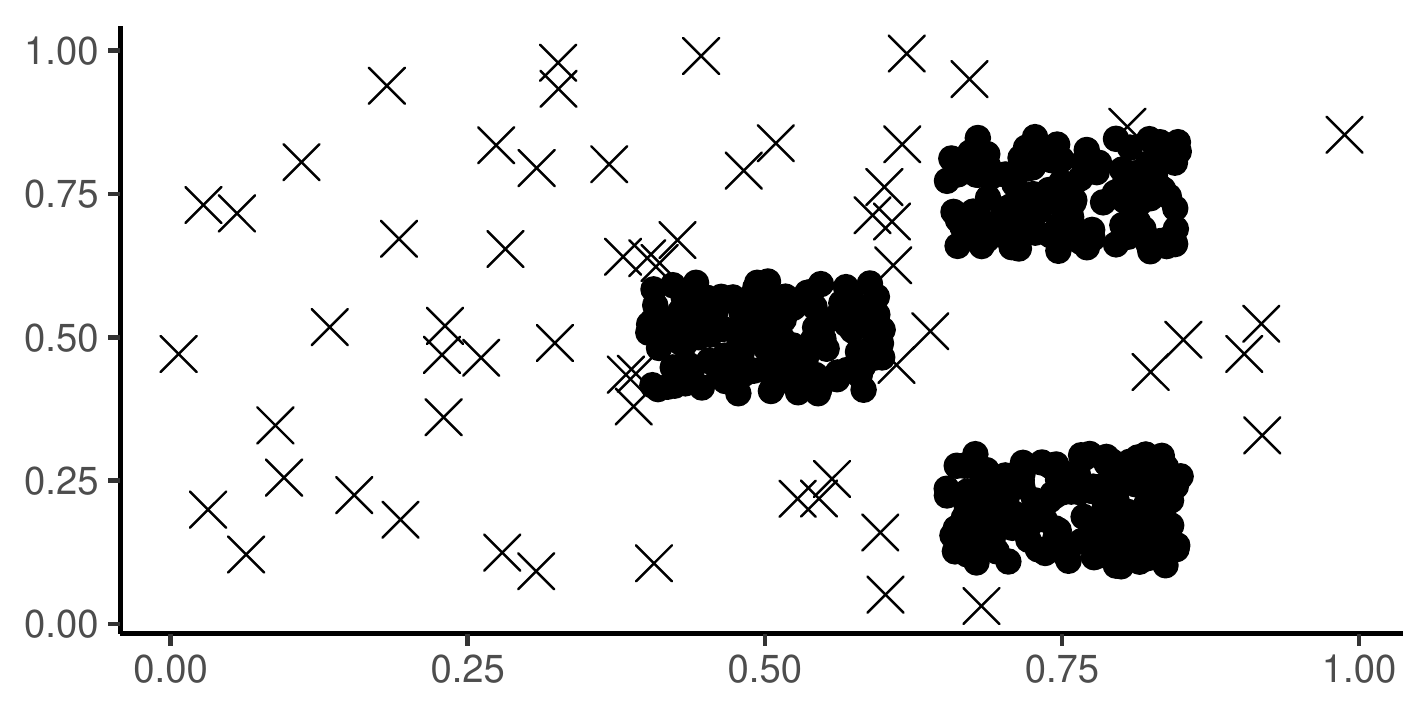}
        \end{tabular}                                                     \\
        \renewcommand{\arraystretch}{1}
        \begin{tabular}{p{2cm}llll}
          \toprule
              & \multicolumn{2}{c}{\textit{ n = 200 }} & \multicolumn{2}{c}{\textit{ n = 500 }}                                       \\[-1.3ex]
              & \multicolumn{2}{c}{\hrulefill}         & \multicolumn{2}{c}{\hrulefill}                                               \\
              & $\varepsilon=0$                        & $\varepsilon=0.2$                      & $\varepsilon=0$ & $\varepsilon=0.2$ \\
          \toprule
          \multicolumn{5}{c}{Case (a): $\delta$ large }                                                                               \\ [1ex]
          OSL & 0.000 (0.000)                          & 0.000 (0.000)                          & 0.000 (0.000)   & 0.000 (0.000)     \\
          SL  & 0.000 (0.000)                          & 0.958 (0.006)                          & 0.000 (0.000)   & 0.997 (0.002)     \\
          SC  & 0.000 (0.000)                          & 0.162 (0.012)                          & 0.000 (0.000)   & 0.145 (0.011)     \\[1ex]
          \multicolumn{5}{c}{Case (b): $\delta$ small}                                                                                \\[1ex]
          OSL & 0.000 (0.000)                          & 0.014 (0.004)                          & 0.000 (0.000)   & 0.002 (0.001)     \\
          SL  & 0.000 (0.000)                          & 1.000 (0.000)                          & 0.000 (0.000)   & 1.000 (0.000)     \\
          SC  & 0.000 (0.000)                          & 0.604 (0.015)                          & 0.000 (0.000)   & 0.469 (0.016)     \\
          \bottomrule
        \end{tabular}                                                     \\
        \includegraphics[width=13cm,height=9cm]{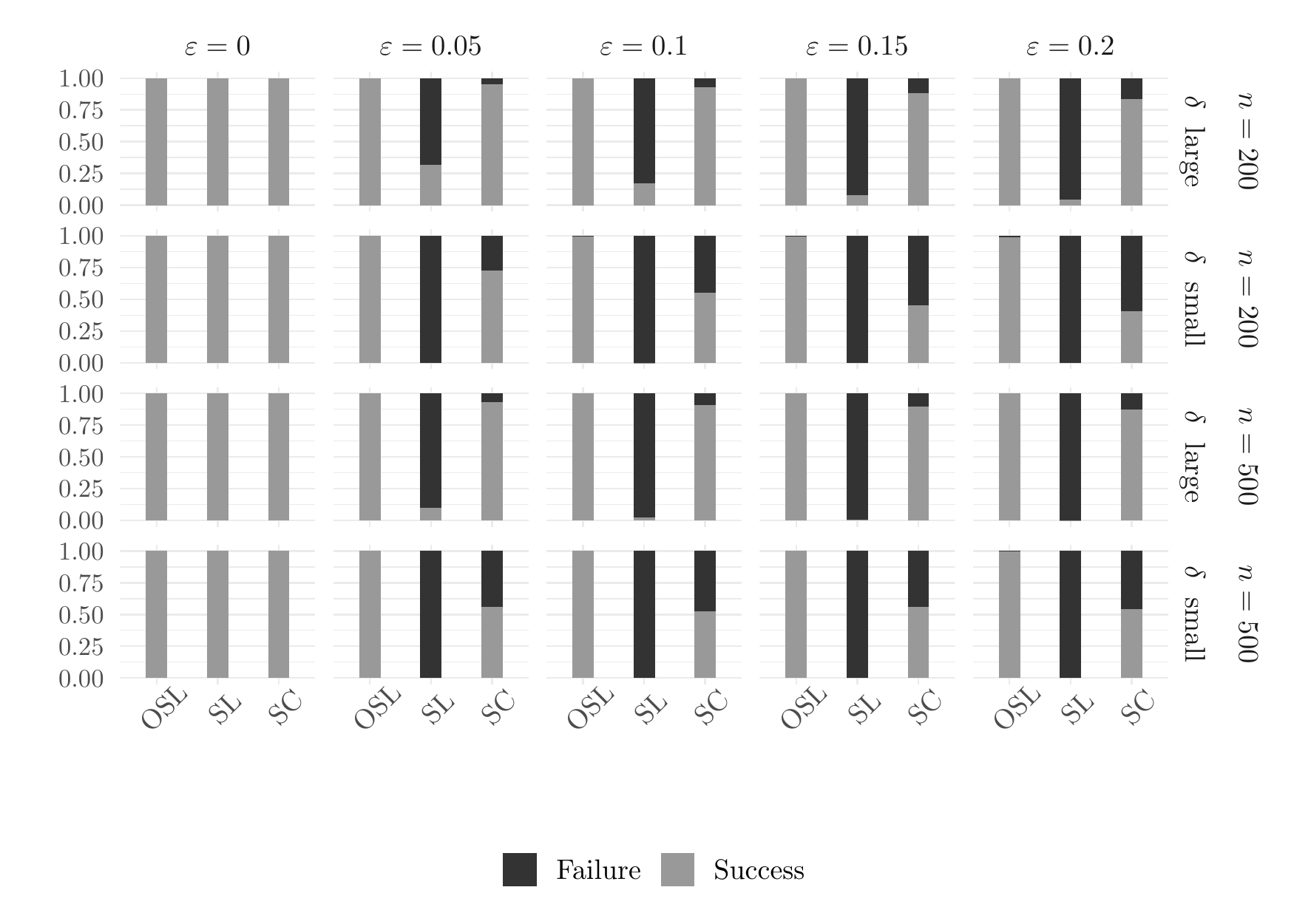} \\
      \end{tabular}}
    \caption{Results in \textit{squares} model. From top to bottom :  a sample of $n=500$ observations with $\varepsilon=0.1$ and (a) $\delta=0.35$  and (b) $\delta =0.07$; a table and a barplot displaying the empirical estimate of the clustering risk according to $\varepsilon$, $n$ and $\delta$.}
    \label{fig:perfs_square}
  }
\end{figure}

{\small
\begin{figure}[hp]

  \centering
  {\footnotesize
    \begin{tabular}{c}
      \renewcommand{\arraystretch}{2}
      \begin{tabular}{cc}
        (a)                                                                      & (b)                                                                    \\
        \includegraphics[width=5cm,height=4cm]{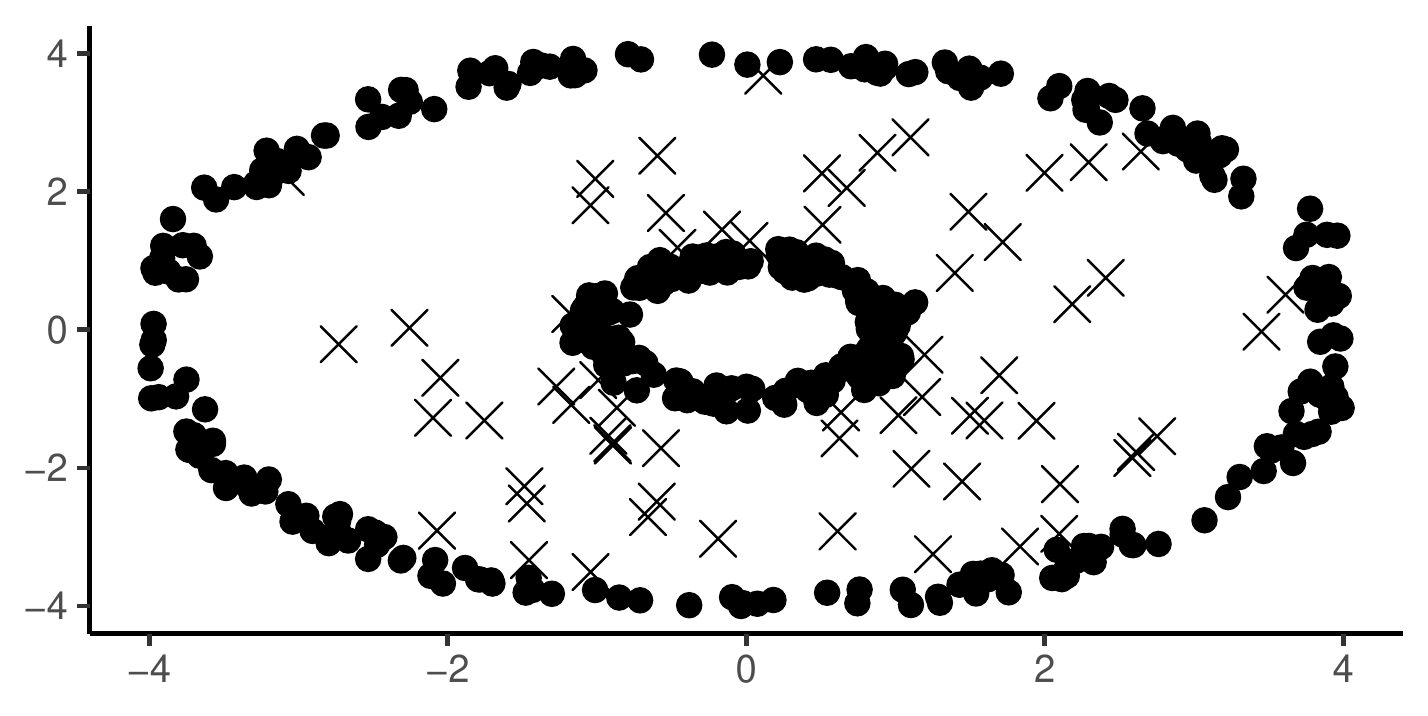} & \includegraphics[width=5cm,height=4cm]{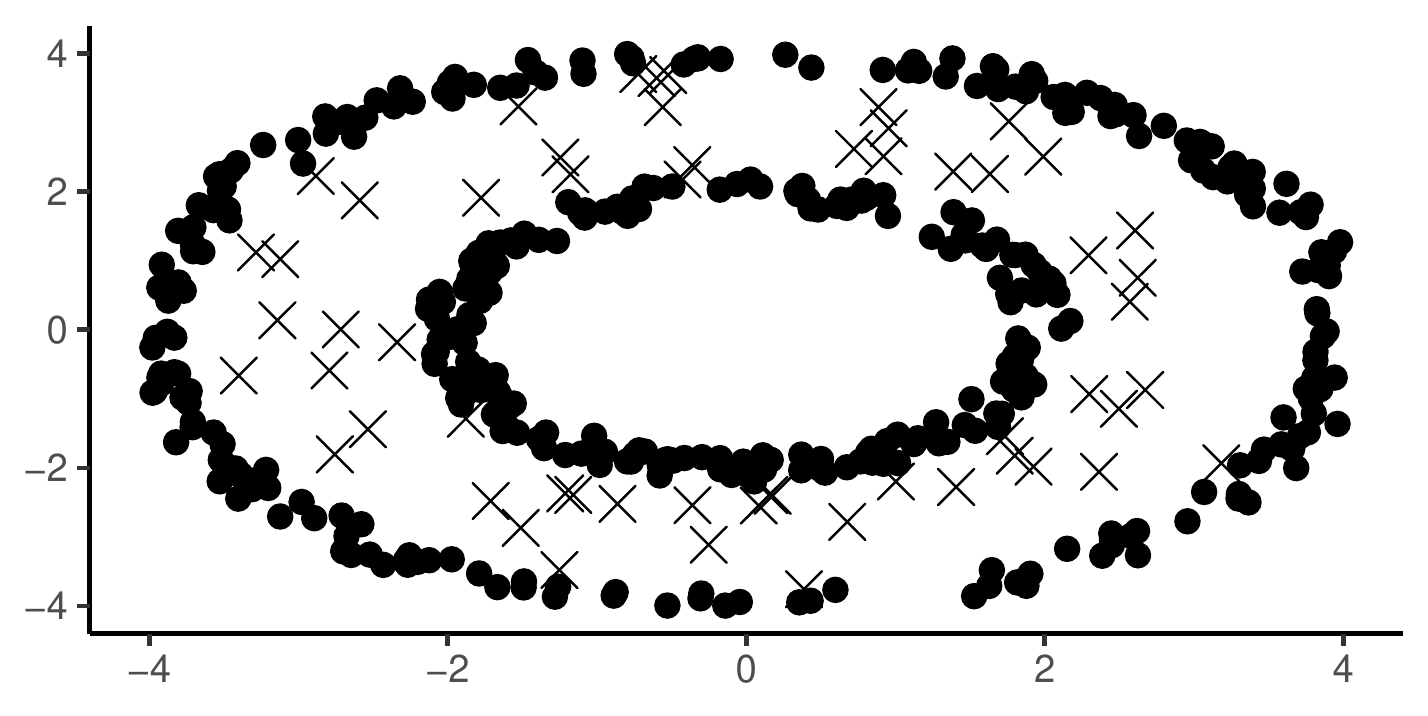} \\
      \end{tabular}                                                     \\

      \renewcommand{\arraystretch}{1}
      \begin{tabular}{p{2cm}llll}
        \toprule
            & \multicolumn{2}{c}{\textit{ n = 200 }} & \multicolumn{2}{c}{\textit{ n = 500 }}                                       \\[-1.3ex]
            & \multicolumn{2}{c}{\hrulefill}         & \multicolumn{2}{c}{\hrulefill}                                               \\
            & $\varepsilon=0$                        & $\varepsilon=0.2$                      & $\varepsilon=0$ & $\varepsilon=0.2$ \\
        \toprule
        \multicolumn{5}{c}{Case (a): $\delta$ large }                                                                               \\ [1ex]
        OSL & 0.001 (0.001)                          & 0.863 (0.011)                          & 0.000 (0.000)   & 0.342 (0.015)     \\
        SL  & 0.000 (0.000)                          & 0.973 (0.005)                          & 0.000 (0.000)   & 0.993 (0.003)     \\
        SC  & 0.000 (0.000)                          & 0.744 (0.014)                          & 0.000 (0.000)   & 0.192 (0.012)     \\[1ex]
        \multicolumn{5}{c}{Case (b): $\delta$ small}                                                                                \\[1ex]
        OSL & 0.001 (0.001)                          & 0.997 (0.002)                          & 0.000 (0.000)   & 0.909 (0.009)     \\
        SL  & 0.000 (0.000)                          & 0.999 (0.001)                          & 0.000 (0.000)   & 1.000 (0.000)     \\
        SC  & 0.001 (0.001)                          & 0.996 (0.002)                          & 0.000 (0.000)   & 0.863 (0.011)     \\
        \bottomrule
      \end{tabular}                                                     \\
      \includegraphics[width=13cm,height=9cm]{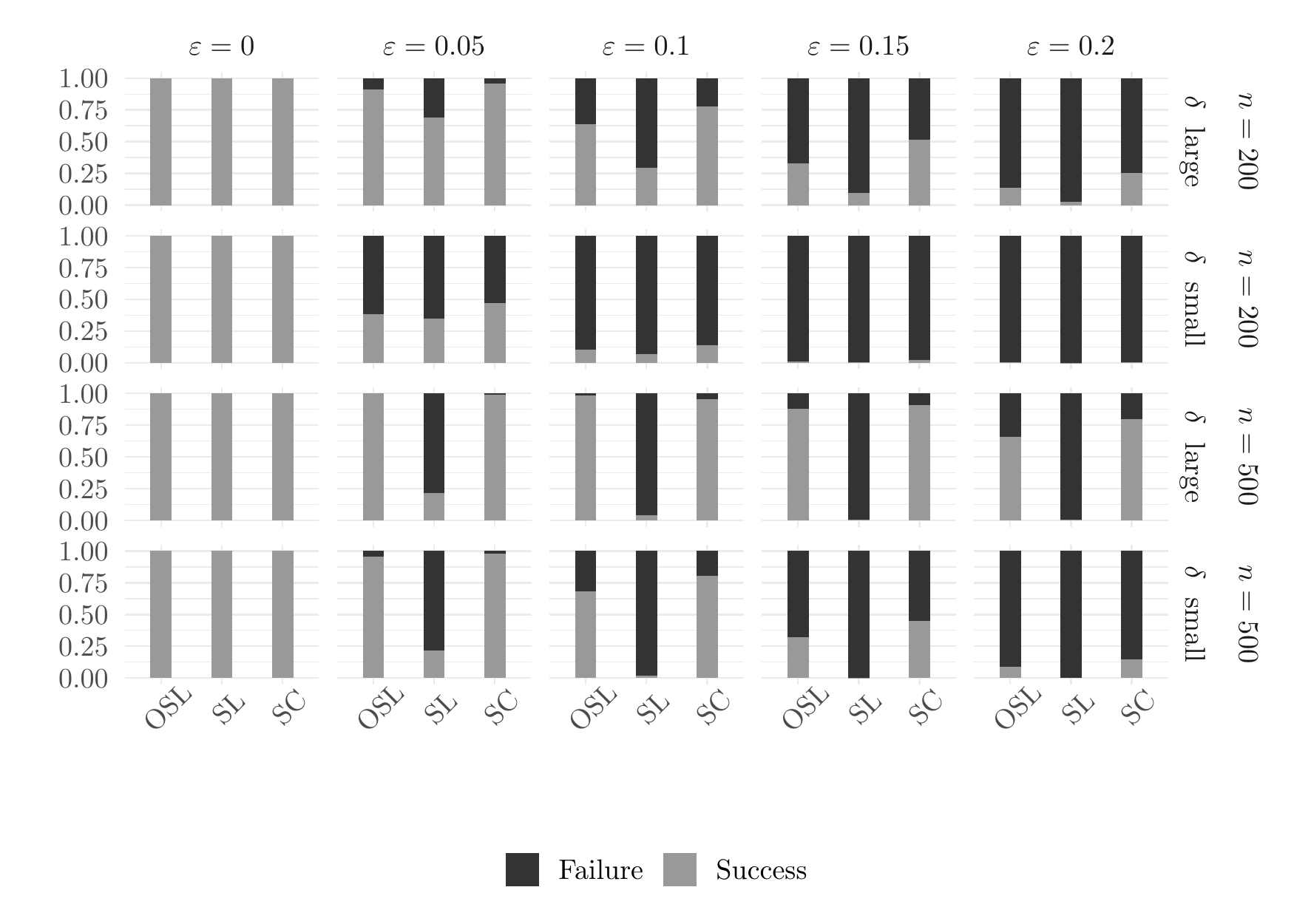} \\
    \end{tabular}	}
  \caption{Results in \textit{concentric circles} model. From top to bottom :  a sample of $n=500$ observations with $\varepsilon=0.1$ and (a) $\delta=2.6$  and (b) $\delta=1.6$;  a table and a barplot displaying the empirical estimate of the clustering risk according to $\varepsilon$, $n$ and $\delta$.}
  \label{fig:perfs_spiral}
\end{figure}}

{\small
\begin{figure}[hp]

  \centering
  {\footnotesize
    \begin{tabular}{c}
      \renewcommand{\arraystretch}{2}
      \begin{tabular}{cc}
        (a)                                                                    & (b)                                                                  \\
        \includegraphics[width=6cm,height=4cm]{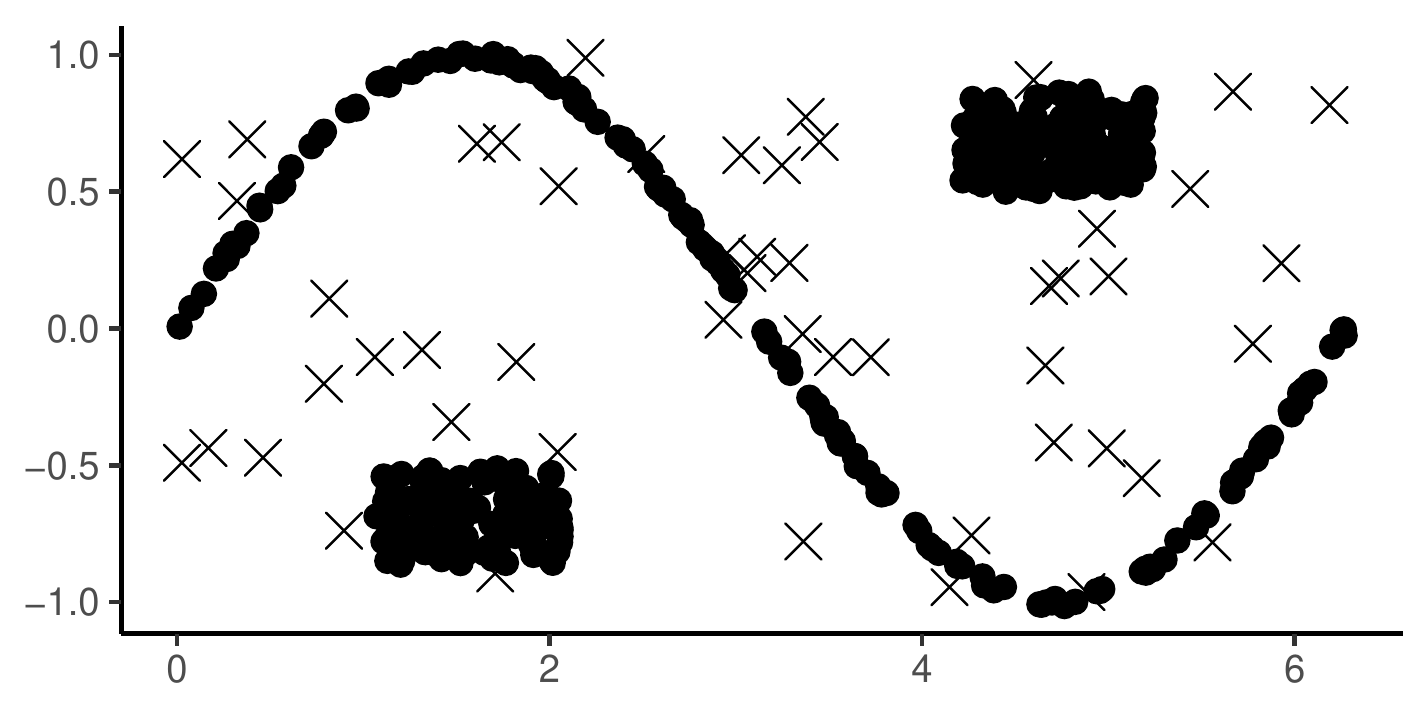} & \includegraphics[width=6cm,height=4cm]{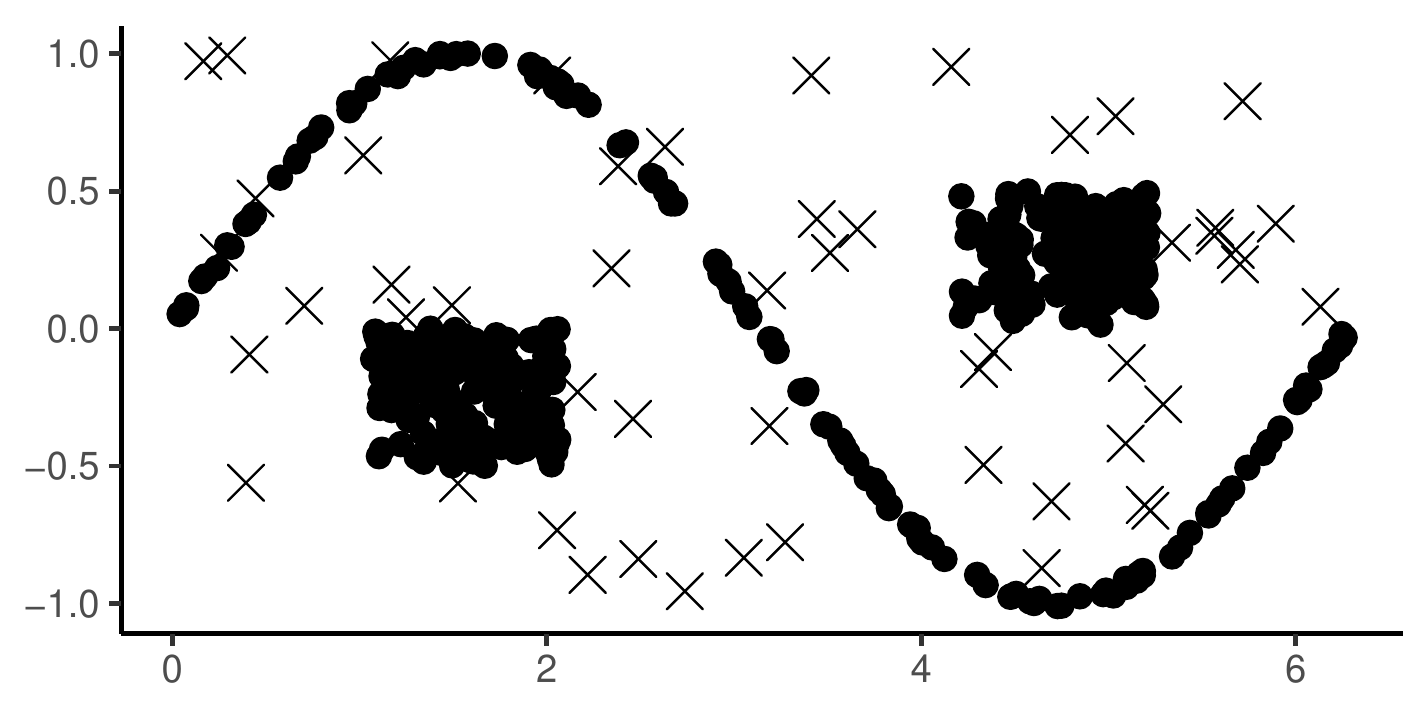} \\
      \end{tabular}                                                    \\
      \renewcommand{\arraystretch}{1}
      \begin{tabular}{p{2cm}llll}
        \toprule
            & \multicolumn{2}{c}{\textit{ n = 200 }} & \multicolumn{2}{c}{\textit{ n = 500 }}                                       \\[-1.3ex]
            & \multicolumn{2}{c}{\hrulefill}         & \multicolumn{2}{c}{\hrulefill}                                               \\
            & $\varepsilon=0$                        & $\varepsilon=0.2$                      & $\varepsilon=0$ & $\varepsilon=0.2$ \\
        \toprule
        \multicolumn{5}{c}{Case (a): $\delta$ large }                                                                               \\ [1ex]
        OSL & 0.001 (0.001)                          & 0.796 (0.013)                          & 0.000 (0.000)   & 0.307 (0.015)     \\
        SL  & 0.001 (0.001)                          & 0.979 (0.005)                          & 0.000 (0.000)   & 1.000 (0.000)     \\
        SC  & 0.009 (0.003)                          & 0.621 (0.015)                          & 0.000 (0.000)   & 0.208 (0.013)
        \\[1ex]
        \multicolumn{5}{c}{Case (b): $\delta$ small}                                                                                \\[1ex]
        OSL & 0.058 (0.007)                          & 0.939 (0.008)                          & 0.000 (0.000)   & 0.548 (0.016)     \\
        SL  & 0.058 (0.007)                          & 0.994 (0.002)                          & 0.000 (0.000)   & 1.000 (0.000)     \\
        SC  & 0.077 (0.008)                          & 0.899 (0.010)                          & 0.002 (0.001)   & 0.502 (0.016)     \\
        \bottomrule
      \end{tabular}                                                    \\
      \includegraphics[width=13cm,height=9cm]{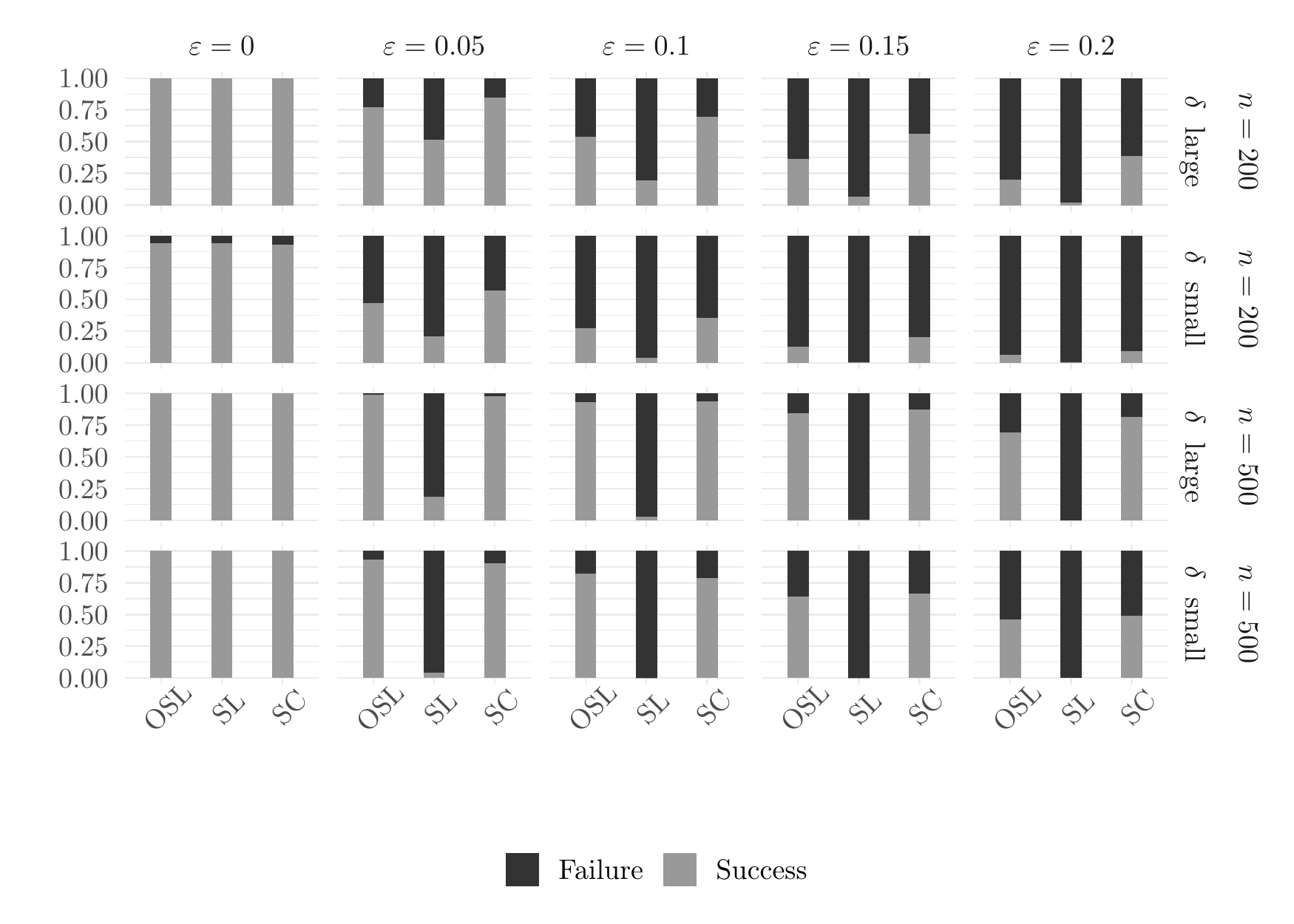} \\
    \end{tabular}	}
  \caption{Results in \textit{sine} model. From top to bottom :  a sample of $n=500$ observations with $\varepsilon=0.1$ and (a) $\delta=1.18$  and (b) $\delta=0.76$; a table and a barplot displaying the empirical estimate of the clustering risk according to $\varepsilon$, $n$ and $\delta$.}
  \label{fig:perfs_sinus}
\end{figure}}

First, as expected, the estimated clustering risk behaves as an increasing function of $\varepsilon$ in all experiments and for all clustering algorithms. Moreover, for a fixed value of $\varepsilon$, the clustering risk of all methods is higher when both the number $n$ of observations and the intergroup distance $\delta$ take small values. In all experiments, when there is no outlier (i.e. $\varepsilon =0$), the clustering risk of both SL and OSL is roughly zero.
This result is consistent with Theorem~\ref{theo:clustriskhatrn}, as well as with the results stated in~\cite{arias2011clustering} and~\cite{auklro15}. In these papers, the authors prove that, under assumptions close to \textbf{(A1)}-\textbf{(A6)} and when $\varepsilon=0$, SL is consistent and its clustering risk tends quickly to zero.
As discussed in Section~\ref{sec:SL}, in presence of outliers SL often fails to recover the true clusters and its clustering risk increases quickly as $\varepsilon$ grows. On the contrary, OSL seems less sensitive to the outlier proportion $\varepsilon$. In the numerical experiments, OSL always works better than SL when there are some outliers.

As expected, SC is well adapted to non-linearly separable groups but can work quite badly with compact groups \citep[see][]{nadler2007fundamental}. Contrary to SC, OSL is competitive in each experiment, and so it seems to perform well whatever the shape of the different groups. Moreover, observe that compared to SC, OSL and SL are exact in the sense that they do not require any random process (for instance the random starts used in SC).

Finally, Table~\ref{fig:computing_time} displays the computation time required by OSL and SC in the \textit{squares} model for various values of $n$. The computations have been carried out on a MacBook Pro, 2,4 GHz Intel Core i5 and 16Gb of RAM memory. Table~\ref{fig:computing_time} shows that OLS is substantially faster than SC as $n$ increases.

\begin{table}[h]
  \centering
  \renewcommand{\arraystretch}{1.1}
  \begin{tabular}{lcc}
    \toprule
    $n$   & OSL  & SC             \\
    \midrule
    100   & 1.41 & 0.10           \\
    200   & 1.56 & 0.38           \\
    500   & 2.09 & 3.86           \\
    1000  & 3.37 & 27.80          \\
    2000  & 3.27 & $2\times 10^3$ \\
    5000  & 3.94 & $4\times 10^4$ \\
    10000 & 6.68 & $3\times 10^5$ \\
    \bottomrule
  \end{tabular}
  \caption{Time complexity in seconds to perform OSL and SC as a function of $n$ in the \textit{squares} model.}
  \label{fig:computing_time}
\end{table}

\subsubsection{Performance according to the distribution of the outliers}

In the previous simulations, we only consider situations where $D=d$ and the outliers are uniformly distributed over their support. In this subsection, we will consider scenarios for which one or both conditions are not satisfied. We first address situations where $D\geq d$. To do so, we consider the tricky case of the sine model for many values of $D$ and $n$. For each couple of values $(D,n)$, outliers are uniformy generated in
\begin{equation*}
  [0,2\pi]^D \setminus \bigcup_{i=1}^MS_{i},
\end{equation*}
while the supports of the three groups remain unchanged (the dimension of the supports of the two compact groups is $2$ while that of the support of the sine group equals 1).
Table~\ref{tab:clust_risk_D_grand} displays the clustering risk of OSL according to $D$ and $n$.

\begin{table}[H]
  \centering
  \renewcommand{\arraystretch}{1.1}
  \begin{tabular}{cccccccccc}
    \toprule
            & \multicolumn{9}{c}{$D$}        \\
            \cmidrule{2-10}
            $n$ & 2     & 3     & 4     & 5     & 6     & 7     & 8     & 9     & 10    \\ 
    \midrule
    100                & 0.984 & 0.928 & 0.876 & 0.848 & 0.808 & 0.786 & 0.784 & 0.784 & 0.791 \\
    200                & 0.929 & 0.672 & 0.405 & 0.253 & 0.198 & 0.160 & 0.161 & 0.163 & 0.161 \\
    500                & 0.558 & 0.080 & 0.013 & 0.002 & 0.000 & 0.000 & 0.000 & 0.000 & 0.000 \\
    1000               & 0.153 & 0.000 & 0.000 & 0.000 & 0.000 & 0.000 & 0.000 & 0.000 & 0.000 \\
    \bottomrule
  \end{tabular}
  \caption{Clustering risk (averaged over 1000 replications) of OSL as a function of $D$ (columns) and $n$ (rows) for the sine model (tricky case with $\varepsilon=0.20$).}
  \label{tab:clust_risk_D_grand}
\end{table}
As proved in Theorem~\ref{theo:comb-th1-th2}, we observe a fast decrease of the risk when $D$ or $n$ increases.  Moreover, for small values of $n$ the risk reaches a minimum value which can be viewed as an ``optimal risk'', i.e. the lowest possible risk for these (small) numbers of observations. In this case, difficulties do not come from the outliers, but it is rather explained by the fact that we do not have enough observations to recover correctly the clusters.

We then address situations where the outliers are not uniformly distributed over their support $S_{0}$.  We consider again the tricky case of the sine model, but we assume that the outliers are densely distributed between the two squares. More precisely, $\prob_0$ is a Gaussian law with mean $(\pi,0)$ and covariance matrix with variances $2\sigma^2$ on the $x$-axis, $\sigma^2$ on the $y$-axis and correlation coefficient $\rho$. Many values for $\sigma^2$ and $\rho$ are considered. Figure~\ref{fig:sinus_bruit_gaussien} displays several examples. Note that $\prob_0$ is truncated in order to avoid that outliers fall into the supports of the clusters. Performances of OSL are given in Table~\ref{tab:clust_risk_bruit_gaussien} for $n=500$ and $\varepsilon=0.1$.
\begin{figure}
  \centering
  \includegraphics{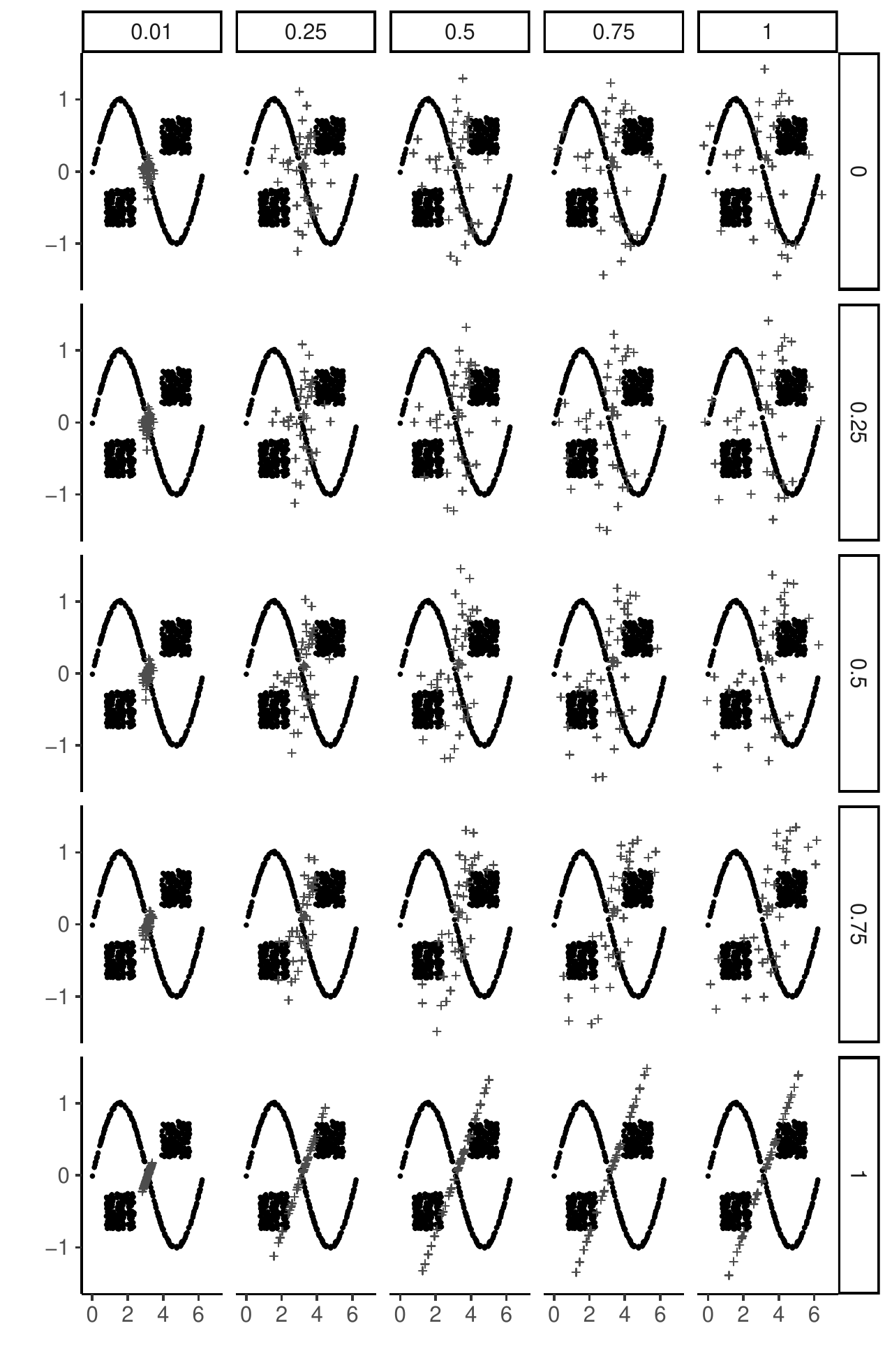}
  \caption{Samples with Gaussian noises in terms of $\sigma^2$ (colums) and $\rho$ (rows).}
  \label{fig:sinus_bruit_gaussien}
\end{figure}


\begin{table}[H]
  \centering
  \renewcommand{\arraystretch}{1.1}
  \begin{tabular}{cccccc}
    \toprule
            & \multicolumn{5}{c}{$\sigma^2$}        \\
            \cmidrule{2-6}
    $\rho$  & 0.01  & 0.25  & 0.5   & 0.75  & 1     \\
    \midrule
    0                             & 0.010 & 0.618 & 0.348 & 0.227 & 0.160 \\
    0.25                          & 0.015 & 0.683 & 0.419 & 0.257 & 0.169 \\
    0.5                           & 0.026 & 0.747 & 0.485 & 0.298 & 0.201 \\
    0.75                          & 0.029 & 0.847 & 0.562 & 0.408 & 0.311 \\
    1                             & 0.037 & 0.982 & 0.964 & 0.920 & 0.869 \\
    \bottomrule
  \end{tabular}
  \caption{Clustering risk (averaged over 1000 Monte Carlo replications) of OSL as a function of $\rho$ (rows) and $\sigma^2$ (columns) for the sine model (tricky case with $n=500$ and $\varepsilon=0.10$).}
  \label{tab:clust_risk_bruit_gaussien}
\end{table}

Some comments can be made about Table~\ref{tab:clust_risk_bruit_gaussien}.
First, we notice that the procedure is efficient for large values of $\sigma^2$ (except for the particular value $\rho=1$). This is simply because  large values of $\sigma^2$ provide sparse outliers, so that the procedure identifies correctly the three clusters. This is no longer the case when $\sigma^2$ decreases since the error term is then increasing. In particular, the algorithm is not efficient when $\sigma^2=0.25$. For this value, the distribution of the outliers is so dense between the square clusters that a path appears between these clusters and the procedure fails to correctly identify the groups.
Next, when $\sigma^2$ becomes much smaller, the error term becomes very small. Indeed, when $\sigma^2=0.01$, the outliers are gathered around the sine curve and prevent the two squares from being connected. In this case, the outliers are assigned to the group formed by the sine curve.
Since the clustering risk measures the ability of a statistical procedure to correctly group observations that belong to true clusters only (and ignores how outliers are assigned), this quantity is not affected by assigning outliers to a group. That's why, the error term remains small in this case.
Lastly, we can remark that the risk increases as $\rho$ becomes larger. Indeed, for large values of $\rho$, the outliers are gathered around a segment that connects the two squares. These two groups are thus connected, and the procedure fails with high probability.

\begin{rem}
  The last scenario with $\sigma^2=0.01$ needs to be discussed a bit further. Indeed, when $\sigma^2=0.01$, outliers are densely grouped around $(\pi,0)$ (see the first column in Figure~\ref{fig:sinus_bruit_gaussien}). So, it seems not easy to differentiate outliers from groups and one could consider that outliers define a group while the sine curve represents the outliers. However, based on the definition of our model, this is not the case. Indeed, let us recall that our model identifies groups in terms of both density in the supports (assumption \textbf{A3}) and cluster size (assumption \textbf{A6}). As the sine wave satisfies these two properties, it must define a group. So OLS procedure is correct when it identifies the sine curve as a true cluster.
\end{rem}

\begin{rem}
  Another issue is the robustness of the results depending on the choice of the number of clusters \citep[see][]{corhen17}. Since outliers may be interpreted as a true group in the last scenario when $\sigma^2=0.01$, we can be interested in the behavior of the algorithm for $M=4$. However, it is easy to see that the sampling design does not satisfy the model assumptions for $M=4$. Indeed, the supports of the sine curve and the outliers intersect so that it is no longer possible to identify 4 groups. Running OSL with $M=4$ thus identifies outliers as a true cluster while sine observations are considered as either observations of the fourth group or outliers. A correct scenario with $M=4$ can be obtained by adding a small separation between the sine group and the outliers. Figure~\ref{fig:sinus_bruit_gaussienM4} displays results of OSL with such a small gap for many values of $M$. For $M=3$, the algorithm actually identifies outliers as a true cluster while the two parts of the sine curve are considered as outliers. One part of the sine curve becomes the fourth group for $M=4$. The other part corresponds to the fifth group for $M=5$. Observe that no observations are identified as outliers for $M=5$.
\end{rem}

\begin{figure}
  \centering
  \includegraphics{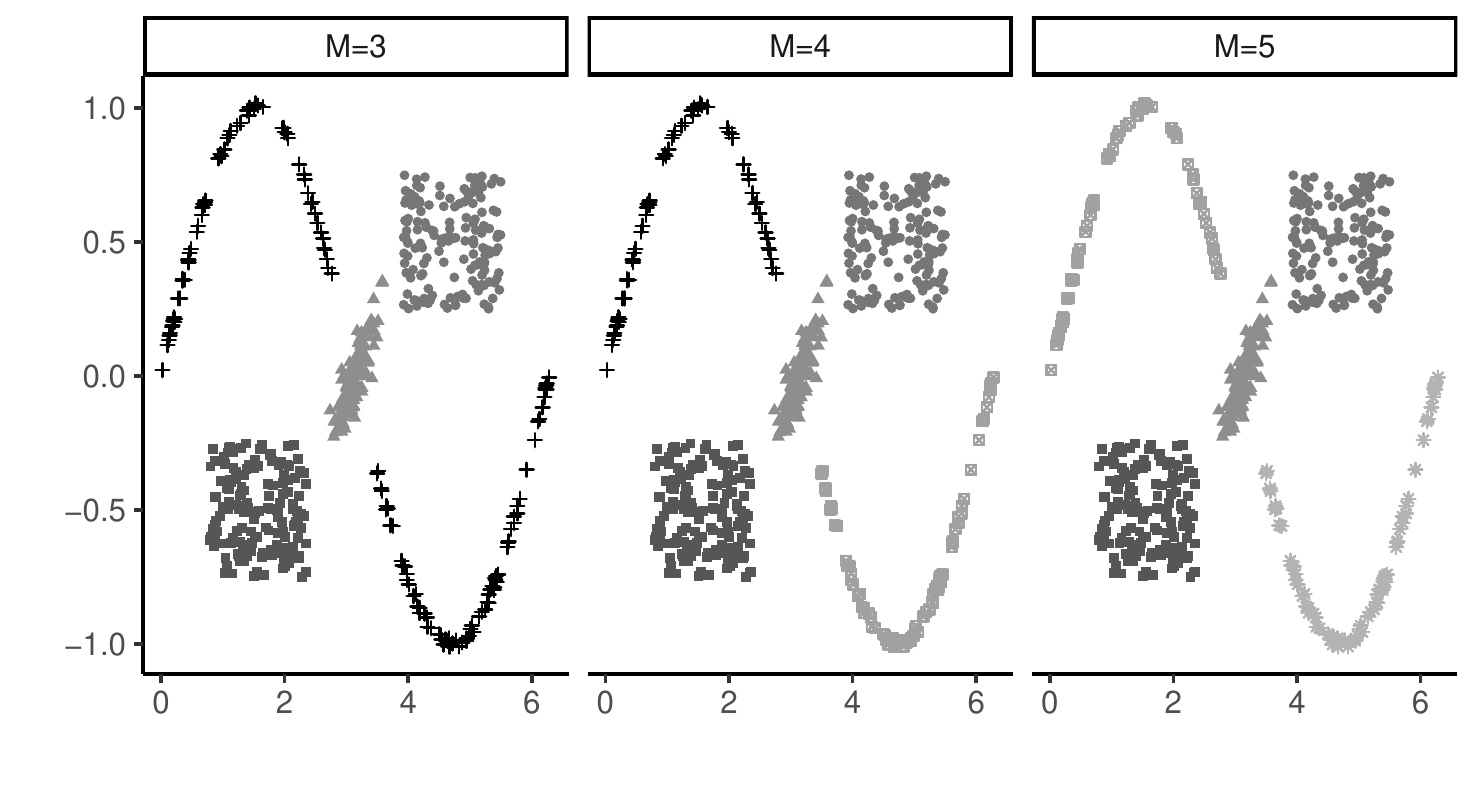}
  \caption{Results for many values of $M$ with $\varepsilon=0.2$, $\sigma^2=0.01$ and $\rho=0.85$. Outliers identified by OSL correspond to ``+''.}
  \label{fig:sinus_bruit_gaussienM4}
\end{figure}

\subsection{Comparison with several common clustering algorithms}

Here, OSL is compared with SL, SC and additional clustering algorithm namely the $k$-means algorithm (KMeans), the trimmed $k$-means (TKMeans), the density-based spatial clustering of applications with noise algorithm (DBSCAN) and the hierarchical density-based spatial clustering of applications with noise (HDBSCAN) based on three datasets from literature.

\subsubsection{Description of the data}

In the three datasets downloaded from \url{https://github.com/deric/clustering-benchmark}, data lie into subspace of $\mathbb R^2$ and groups have various shapes and various sizes. Moreover, the proportion and the distribution of outliers is not the same in the three datasets.

\begin{itemize}
  \item In the first scenario titled \textit{pathbased} and used in~\cite{chang2008robust}, the 300 observations belong to three disjoint groups with various shapes. The first two groups are compact with similar areas. There are surrounded by a third group with a ring-shape. In this scenario there is no outlier.

  \item The second dataset named \textit{cure-t2-4k}  \citep[see][]{van2017constraint,nerurkar2018empirical} includes 4200 observations. 4.7\% of observations are outliers and the remaining observations lie into 4 well-separated groups. The first three groups are circular with various areas: one is very large while the two others are small. The fourth group is defined by two similar circles joined by a segment. Outliers are distributed quite uniformly all around the groups.

  \item In the third dataset named \textit{compound}  \citep[see][]{zahn1971graph} there are 399 observations including 50 outliers (12.5\% of observations). Observations that are not outliers are distributed into 5 compact groups of various shapes. The three first groups are distinct and well separated while the last two groups with circular shape touch each other but do not overlap. Outliers are uniformly distributed only around a group and so lie into a subspace of the input space.

\end{itemize}

The three datasets are displayed in Figure~\ref{fig:perfs_sinus_ari}.
Observe that these three scenarios do not necesarily satisfy all the design conditions stated in Section~\ref{sec:main_results}.
Indeed, in the first scenario, the distance between the ring-shaped group and the two compact groups is almost zero. In the second scenario, two clusters are also not well separated and assumption \textbf{(A6)} is not satisfied: the largest group includes 39.6\% of observations while the smallest one includes only 4\% of observations. In the third dataset, assumption \textbf{(A6)} is also not satisfied: about 42\% of observations belong to the largest group while only  9.5\% of observations are in the smallest one.

\subsubsection{Comparison results}

For each scenario, 1000 Monte Carlo replications are used by sampling each time without replacement $75\%$ of the observations. The adjusted rand index (ARI) introduced by~\cite{hubert1985comparing} and the time complexity (TC) are estimated for each clustering algorithm and each Monte Carlo replication. Here, the ability of each algorithm to identify the groups is measured using ARI. This performance criterion is less strict than the clustering risk defined in equation~\eqref{eq:defclustrisk} in the sense that it is not a binary criterion that considers that a clustering procedure fails as soon as one observation that belongs to one true cluster has not been assigned to the right group.  ARI is a measure of similarity between two partitions. It has a value between 0 and 1, with 0 indicating that the two partitions do not agree on any pair of points and 1 indicating that the two partitions match perfectly. Here, for each clustering algorithm and each Monte Carlo replication, we compare the resulted partition with the true partition. Then, ARI represents the proportion of agreements over all the possible pairs of points between the resulted partition and the true one. This less strict criterion seems more adapted to the three considered datasets in which, as on many real applications, some clusters could not be well separated. So the definition of the true partition is not straightforward and several configurations could be considered for a same dataset, see for instance the descriptions of the \textit{pathbased} dataset and the \textit{compound} dataset above.

Implementation and calibration of OSL, SL and SC are the same as those used in the first simulation studies, see subsection~\ref{sec:simu1}. The function \texttt{stats::kmeans} is used to perform KMeans with 20 distinct random starts (parameter \texttt{nstart}). TKMeans is performed using the function \texttt{tclust::tkmeans} with 50 distinct random starts (the default value of parameter \texttt{nstart}). In TKMeans, the proportion $\alpha$ of trimmed observations is tuned using the true proportion of outliers $\varepsilon$: at each step, before updating the centers, the algorithm removed the top $\lceil \alpha n\rceil = \lceil \varepsilon n\rceil$ observations with the largest distance from its closest center.
DBSCAN and HDBSCAN are performed by using the functions \texttt{dbscan::dbscan} and \texttt{dbscan::hdbscan}. In DBSCAN, to calibrate the two main tunning parameters that are the radius of each neighborhood (\texttt{eps}) and the number of minimum points required for each neighborhood (\texttt{minPts}), we use the approach explained by~\cite{ester1996density}: for each scenario, \texttt{minPts} is set to its default value $4$ and \texttt{eps} is calibrated once on the entire dataset. For HDBSCAN, we fix the number of minimum points (\texttt{minPts}) at the minimal possible value, $2$, and we use the knowledge of $M$ to choose the final partition as we do with the algorithms OSL, SL, KMeans and SC. Note that in presence of outliers, this automatic calibration strategy might not be optimal for HDBSCAN and the selection of the final partition might be rather performed on each run in a non-automatic manner (for instance the study of the silhouette score).
As in the previous simulation study, all computations have been carried out on a MacBook Pro, 2.4\,GHz Intel Core i5 with 16\,Gb of RAM.

  {\small
    \begin{figure}[hp]

      \centering
      {\footnotesize
        \renewcommand{\arraystretch}{2}
        \begin{tabular}{c}
          \begin{tabular}{ccc}
            \textit{Pathbased} dataset                                 & \textit{Cure-t2-4k} dataset & \textit{Compound} dataset \\
            \includegraphics[width=3.7cm,height=3.7cm]{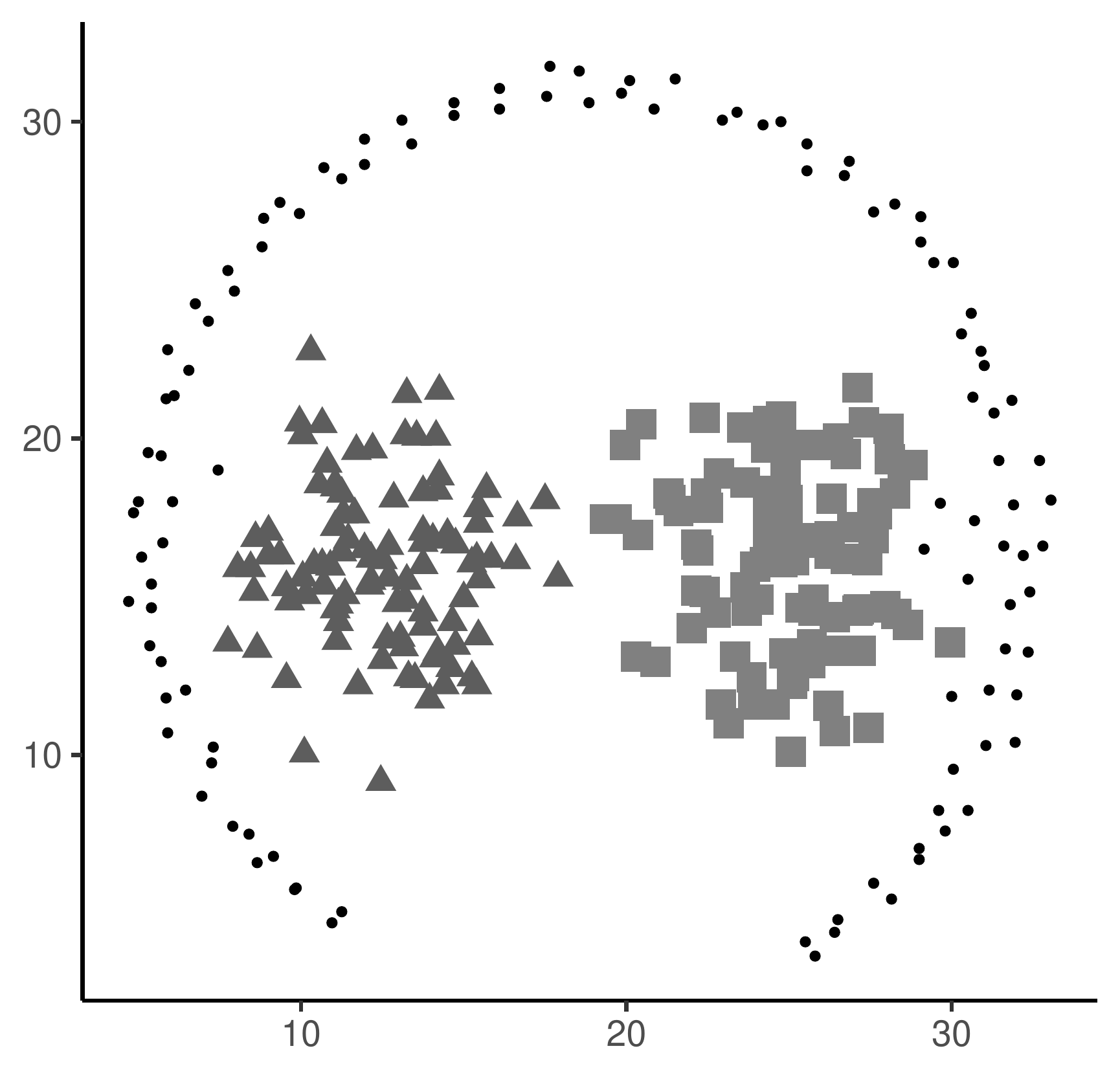}  &
            \includegraphics[width=3.7cm,height=3.7cm]{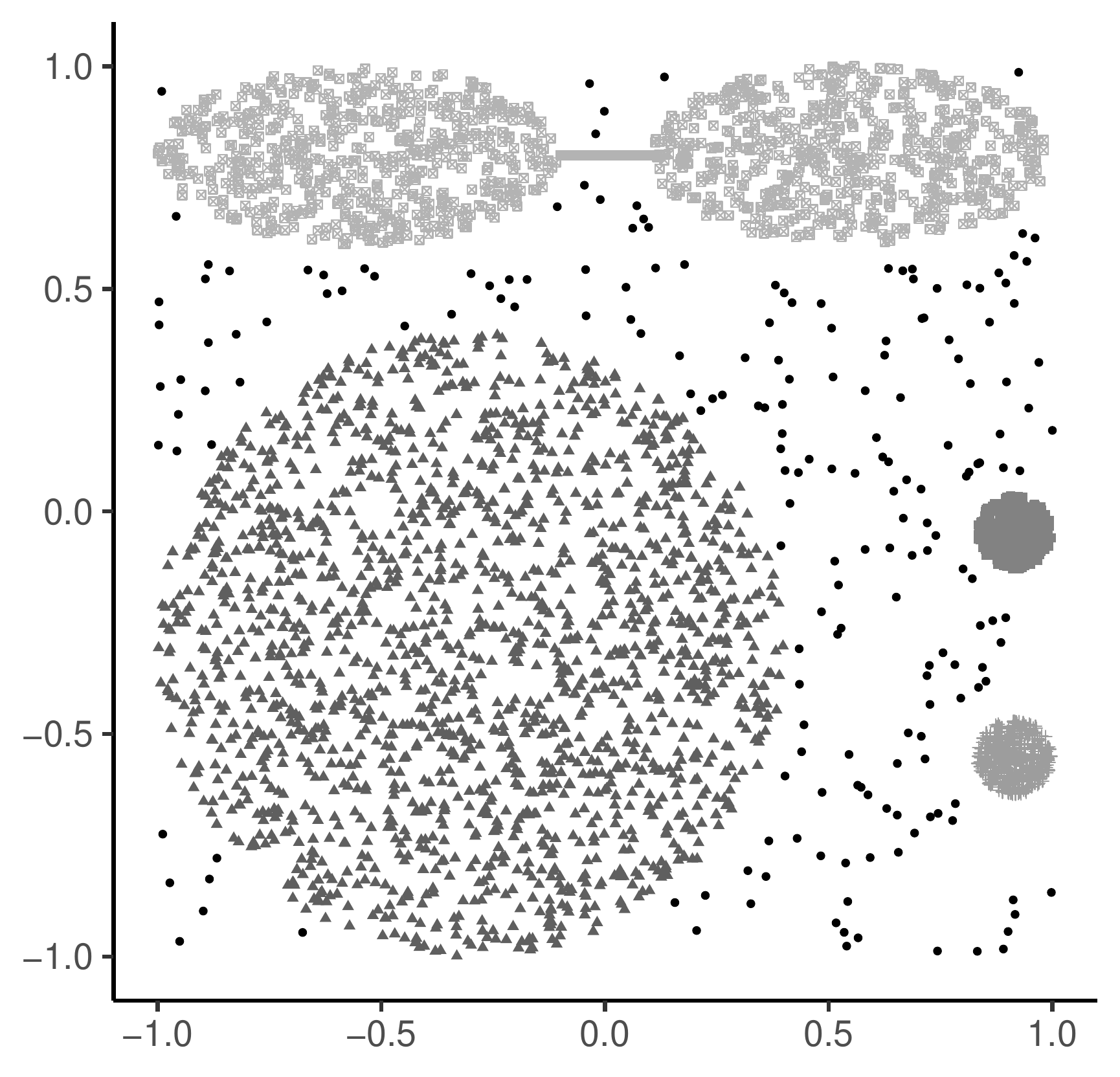} &
            \includegraphics[width=3.7cm,height=3.7cm]{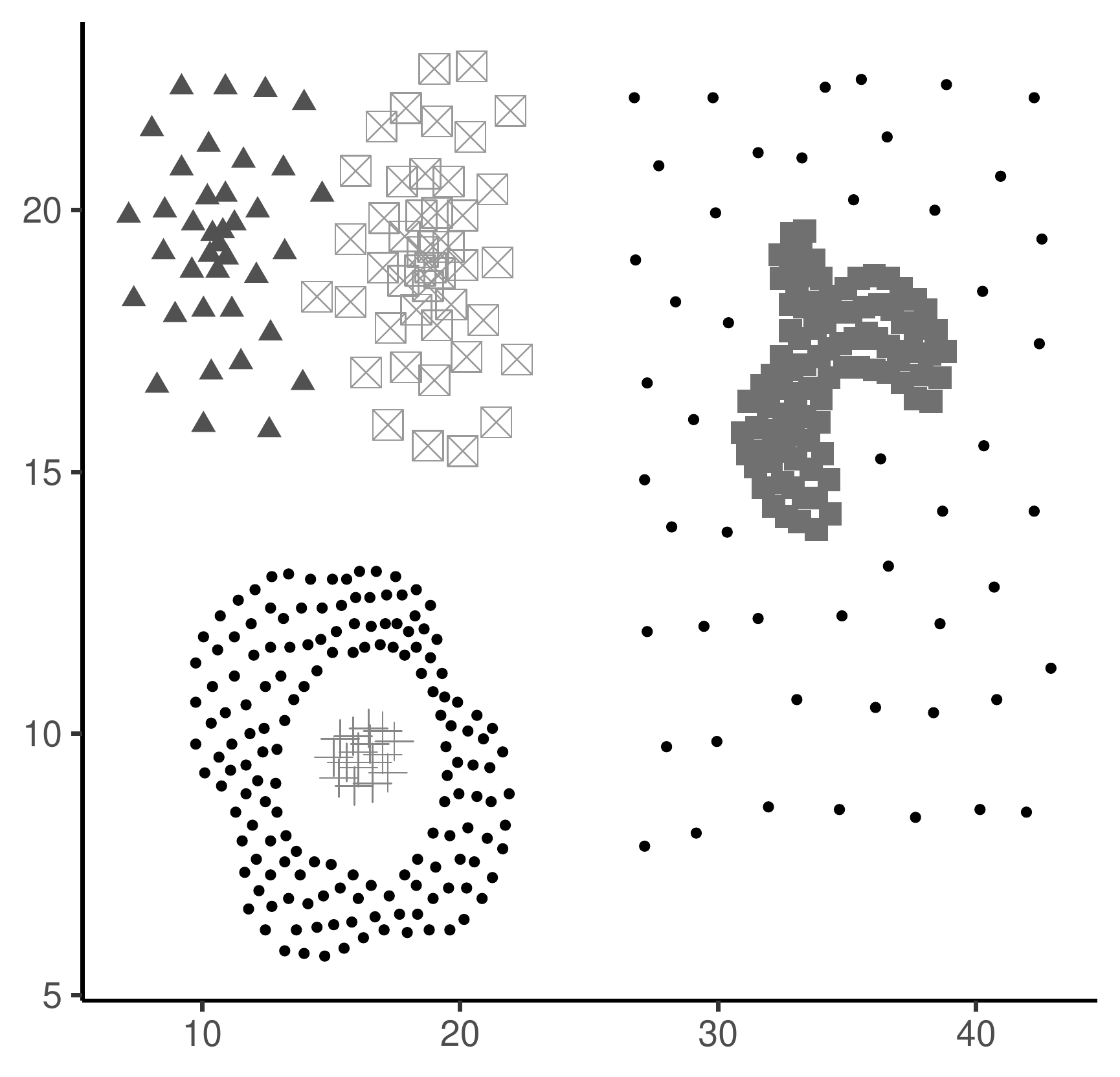}                                                    \\
          \end{tabular}
          \\
          \renewcommand{\arraystretch}{1.2}
          \begin{tabular}{p{4cm}cc}
            \toprule
                    & Adjusted rand index   & Running time (in sec.) \\
            \toprule
            \multicolumn{3}{c}{ \textit{Pathbased} dataset }         \\ [1ex]
            OSL     & \textbf{0.58} (0.16)  & 0.02 (0.01)            \\
            SL      & 0.07 (0.14)           & 5.9e-4  (8.1e-4)       \\
            SC      & 0.39 (0.19)           & 0.05 (0.01)            \\
            KMeans  & 0.46 (0.03)           & 1.8e-3 (1.6e-3)        \\
            TKMeans & \textbf{0.9} (0.04)   & 7.1e-3 (1.2e-3)        \\
            DBSCAN  & \textbf{0.6} (0.06)   & 3.8e-4 (8e-05)         \\
            HDBSCAN & 0.07 (0.14)           & 7.4e-3 (1.8e-3)        \\
                    &                       &                        \\
            \multicolumn{3}{c}{\textit{Cure-t2-4k} dataset}          \\[1ex]
            OSL     & \textbf{0.9} (0.13)   & 1.43 (0.07)            \\
            SL      & 0.02 (0.08)           & 0.06 (0.01)            \\
            SC      & \textbf{0.91} (0.09)  & 98.5 (3.6)             \\
            KMeans  & 0.54 (0.01)           & 0.03 (5.9e-3)          \\
            TKMEANS & 0.53 (0.01)           & 0.27 (0.01)            \\
            DBSCAN  & \textbf{0.96} (0.021) & 3.0e-3 (5.9e-4)        \\
            HDBSCAN & 0.02 (0.08)           & 0.82 (0.05)            \\
                    &                       &                        \\
            \multicolumn{3}{c}{\textit{Compound} dataset}            \\[1ex]
            OSL     & 0.48 (0.33)           & 0.02 (5e-3)            \\
            SL      & \textbf{0.69} (0.12)  & 9.5e-4 (4.1e-3)        \\
            SC      & \textbf{0.72} (0.11)  & 0.10 (9.4e-3)          \\
            KMeans  & 0.58 (0.04)           & 2.3e-3 (6.4e-4)        \\
            TKMEANS & 0.61 (0.02)           & 0.01 (1.3e-3)          \\
            DBSCAN  & 0.32 (0.08)           & 4e-04 (4e-05)          \\
            HDBSCAN & \textbf{0.69} (0.12)  & 8.3e-3 (1.2e-3)        \\
            \bottomrule
          \end{tabular} \\
        \end{tabular}	}
      \caption{\small Results on the \textit{Pathbased}, \textit{Compound} and \textit{Aggregation} datasets. For each parameter, the mean over the 1000 Monte Carlo replications is displayed with in brackets the standard error.}
      \label{fig:perfs_sinus_ari}
    \end{figure}}

Results are displayed in Figure~\ref{fig:perfs_sinus_ari}.
First, we can observe that OLS generally performs quite well to recover the true groups. The method is part of the top three algorithms in the two first scenarios. Note that in the two first scenarios, OLS outperforms SL even when there is no outlier.

In the \textit{Pathbased} dataset and particularly in the \textit{Compound} dataset, performance of OLS is affected by the fact that some clusters are not well separated and the distance between some clusters is almost zero. Indeed, for instance in the \textit{Compound} dataset, clustering performance of OLS improves greatly when we consider the partition in which the two groups that intersect are grouped together, see Figure~\ref{fig:perfs_sinus_ari_2}. Remark that in this case OLS outperforms all the other clustering algorithms.

Generally, our clustering method is thus competitive compared to the other algorithms. As previously noticed, OLS does not seem very sensitive to the group shape. As DSBCAN, OLS seems able to identify group completly surrounded by another group and seems quite robust toward outliers. Moreover, note that contrary to DBSCAN, OLS seems less sensitive to datasets with large density variations between groups (see for instance the \textit{Compound} dataset). Indeed, DBSCAN is very sensitive to the choice of its two main tunning parameters (\texttt{MinPts} and \texttt{eps}) and these two parameters cannot be chosen appropriately for all clusters when the density varies a lot between clusters.

Finally, OLS is compared to the other algorithms in terms of time complexity. OLS appears a bit less fast than SL, KMeans, TKMeans and DBSCAN. Nonetheless, the time complexity of our apporach remains raisonnable when $n$ is larger. Moreover, as peviously noticed (see Table~\ref{fig:computing_time}), OLS is faster than SC, especially when $n$ is large.

\begin{figure}[h]
  \begin{tabular}{cc}
    \begin{minipage}{0.48\linewidth}
      \includegraphics[width=6cm,height=5cm]{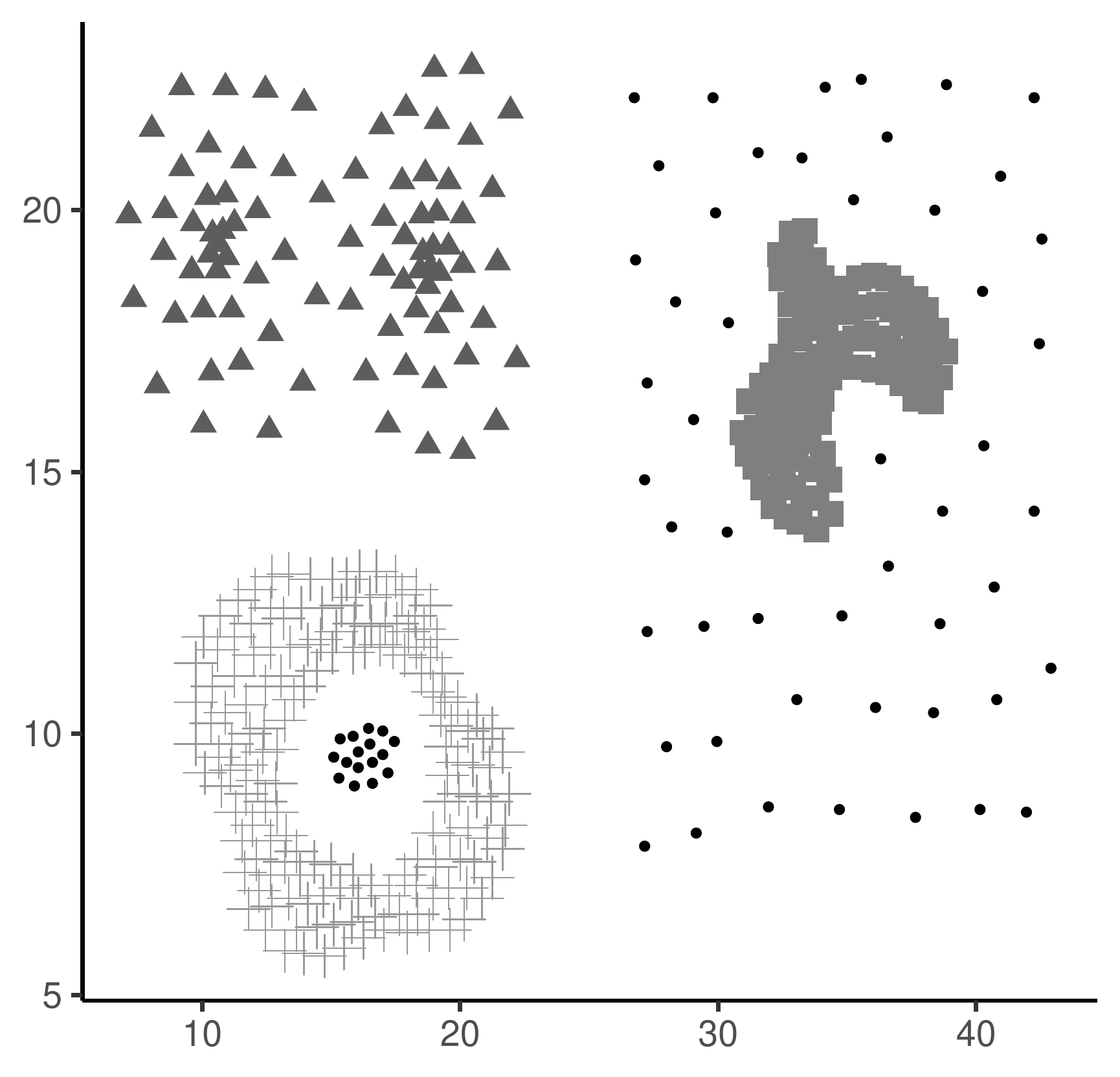}
    \end{minipage}
     &
    \begin{minipage}{0.48\linewidth}
      \begin{tabular}{p{2cm}cc}
        \toprule
                & Adjusted             & Running         \\
                & rand index           & time (s)        \\
        \toprule
        OSL     & \textbf{0.7} (0.22)  & 2.5e-3 (2.5e-3) \\
        SL      & 0.66 (0.16)          & 7.7e-4 (4.5e-4) \\
        SC      & \textbf{0.75} (0.03) & 0.10 (8.9e-3)   \\
        KMeans  & 0.65 (0.07)          & 2.2e-3 (4.3e-3) \\
        TKMEANS & \textbf{0.67} (0.08) & 0.01 (1e-3)     \\
        DBSCAN  & 0.35 (0.07)          & 4.1e-4 (4e-05)  \\
        HDBSCAN & 0.66 (0.16)          & 8.4e-4 (1.4e-3) \\

        \bottomrule
      \end{tabular}
    \end{minipage}
  \end{tabular}
  \caption{Results on the modified version of the \textit{Compound} datasets. For each parameter, the mean is displayed with in brackets the standard error.}
  \label{fig:perfs_sinus_ari_2}

\end{figure}

\section{Proofs}\label{sec:proofs}

\subsection{Technical lemmas}

\begin{lemma}
  \label{lem:within-cluster-connectivity}
  Fix $i=1,\dotsc,M$ and $0< r < \Delta_i$. Under {\bf (A1)-(A3)-(A4)}, there exists a positive constant $\Lambda_i$ such that
  \begin{equation*}
    \psi_{n,i}(r)=\prob(\mathbb{X}_n\cap S_{i}\text{ is not $r$-connected})\leq \Lambda_i r^{-d} \exp(-\constA nr^{d}).
  \end{equation*}
\end{lemma}

\begin{proof}[\textbf{\upshape Proof:}]
  We denote by $N_r^*(S_{i})$ the minimal number of balls of radius $r>0$, centered at points of $S_{i}$, required to cover $S_{i}$. Note that, for any $r>0$ we have by definition $N_r(S_{i}) \leq N_r^*(S_{i})$. Moreover, using triangle inequality we also have $N_r^*(S_{i}) \leq N_{r/2}(S_{i})$. Using Assumption~\textbf{(A5)}, this implies that
  \begin{equation*}
    s_i = \lim_{r\to 0} \frac{\log(N_r^*(S_{i}))}{\log(1/r)}.
  \end{equation*}
  Thus, there exists a positive constant $\Lambda_i$ (that depends on $S_{i}$) such that, for any $0<r<\Delta_i$ we have:
  \begin{equation*}
    N_r^*(S_{i})
    \leq 4^{-d}\Lambda_i r^{-s_i}
    \leq 4^{-d}\Lambda_i r^{-d}.
  \end{equation*}
  This implies that there exist both an index set $\mathcal{L}_i$, whose cardinality is bounded above by $\Lambda_i r^{-d}$, and a family of balls $(B_{\ell})_{\ell\in\mathcal L_i}$ centered at points that belong to $S_{i}$,  with radius $r/4$, which satisfy:
  \begin{equation*}
    S_{i} \subset \bigcup_{\ell\in\mathcal{L}_i} B_\ell.
  \end{equation*}
  Since, for any $\ell\in\mathcal L_i$, we have $(\mathbb{X}_n\cap S_{i})\cap B_\ell\neq \emptyset$, there exists $\alpha_\ell\in\In$ such that $X_{\alpha_\ell}\in B_{\ell}\cap S_{i}$. Using triangle inequality, this implies that $B(X_{\alpha_\ell},r/2)\supset B_\ell$. Thus,
  \begin{equation*}
    \mathbb{X}_n\cap S_{i} \subset \bigcup_{\ell\in\mathcal{L}_i} B(X_{\alpha_\ell}, r/2)
    \quad\text{with}\quad X_{\alpha_\ell} \in \mathbb{X}_n\cap S_{i}.
  \end{equation*}
  We deduce that $\mathbb{X}_n\cap S_{i}$ is $r$-connected and
\begin{align*}
  \psi_{n,i}(r)
   & \leq \prob\left(
  \exists \ell\in\mathcal{L}_i, B_{\ell}\cap (\mathbb{X}_n\cap
  S_{i})=\emptyset
  \right)                                                                                            \\
   & \le \prob\left(
  \exists \ell\in\mathcal{L}_i, \; \forall k\in {\In},
  X_{k}\notin S_{i} \text{ or } (X_{k}\in S_{i}, X_{k}\notin B_{\ell})
  \right)                                                                                            \\
   & \leq \sum_{\ell\in\mathcal{L}_i} \big(
  \prob(X\notin S_{i}) + \prob(X\notin B_{\ell} \mid X\in
  S_{i})\prob(X \in S_{i})
  \big)^{n}                                                                                          \\
   & \leq \sum_{\ell\in\mathcal L_i} \big(
  1 - \prob(X\in B_{\ell} \mid X\in S_{i}) \prob(X\in S_{i})
  \big)^{n}                                                                                          \\
   & \leq \sum_{\ell\in\mathcal L_i} \big(1-(1-\varepsilon)\gamma_i \prob_{i}(X\in B_{\ell})\big)^n.
\end{align*}
Moreover,
\begin{align*}
  \prob_{i}(X\in B_{\ell}) & =\prob_i(X\in B_\ell\cap S_{i})                           & \text{from {\bf (A1)}} \\
                           & \geq \kappa_i^{-1} \hausdorff^{s_{i}}(B_{\ell}\cap S_{i}) & \text{from {\bf (A3)}} \\
                           & \geq (\kappa_i\kappa_c)^{-1}\eta(s_{i}) r^{s_{i}}           & \text{from {\bf (A4)}} \\
                           & \geq (\kappa^*\kappa_c)^{-1} \eta_{*}(d) r^{d},             &
\end{align*}
where $\kappa^*$ and $\eta_{*}(d)$ are defined by~\eqref{eq:kappa-eta-star}. Putting all pieces together we obtain
\begin{align*}
  \psi_{n,i}(r) & \leq |\mathcal L_i|
  \big( 1-(1- \varepsilon)\gamma_*(\kappa^*\kappa_c)^{-1} \eta_{*}(d) r^{d} \big)^{n} \\
                & \leq \Lambda_i r^{-d} \exp(-\constA nr^{d}),
\end{align*}
where $\constA$ is defined in~\eqref{eq:def-constAB}.
\end{proof}
\begin{lemma}
  \label{lem:between-cluster-connectivity}
  Let $r>0$ and denote by $\varphi_{n}(m, r)$ the probability that
  there exists, in $S_{0}$, a path of at least $m$ $r$-connected observations. If assumptions {\bf (A1)} and {\bf (A2)} hold, we have
  \begin{equation*}
    \varphi_{n}(m,r) \leq n\varepsilon (\constB n \varepsilon r^{\dd})^{m-1},
  \end{equation*}
  where $\constB$ is defined in~\eqref{eq:def-constAB}.
\end{lemma}

\begin{proof}[\textbf{\upshape Proof:}]
  Fix $r>0$. For
  any $I\subseteq\In$ we denote by $\mathcal A_{I}$ the following event:
  there exists a permutation $i_{1}<\dotsc<i_{m}$ of $I$ such that
  $\|X_{i_{j}}-X_{i_{j+1}}\|\leq r$ for any $j=1,\dotsc,m-1$. We have:
  \begin{align*}
    \varphi_{n}(m,r)
     & \leq \sum_{\substack{I\subseteq\In \\|I|=m}}
    {\prob}(\mathcal A_{I}\cap\{X_{I}\subseteq
    S_{0}\})                            \\
     & \leq \sum_{\substack{I\subseteq\In \\|I|=m}} \varepsilon^{m}
    {\prob_{0}}(\mathcal A_{I}).
  \end{align*}
  Now remark that
  \begin{align*}
  {\prob_{0}}(\mathcal A_{I})
   & \leq m!{\esp_{0}}\left(\prod_{j=1}^{m-1}\1_{[0,
  r]}(\|X_{i_{j}}-X_{i_{j+1}}\|)\right)                                       \\
   & = m!\int_{S_{0}}\cdots\int_{S_{0}} \1_{[0,r]}(\|x_{1}-x_{2}\|)\dotsc
  \1_{[0,r]}(\|x_{m-1}-x_{m}\|)
  \mathrm{d}{\prob_{0}}(x_{1},\dotsc,x_{m}).
\end{align*}
  Note also that, using \textbf{(A2)}:
  \begin{equation*}
    \int_{S_{0}} \1_{[0,r]}(\|x-y\|) \mathrm{d}\prob_{0}(y)
    \leq \prob_{0}(B(x, r)) \leq \kappa_0 \hausdorff^{\dd}(B(x,r)) =
    \kappa_0 \eta(\dd) r^{\dd}.
  \end{equation*}
  This, combined with Fubini's theorem implies that:
  \begin{align*}
    {\prob_{0}}(\mathcal A_{I})
     & \leq m! (\eta(\dd)\kappa_0 r^{\dd})^{{m-1}}.
  \end{align*}
  Finally, we obtain:
  \begin{align*}
  \varphi_{n}(m,r)
   & \leq \frac{n!}{(n-m)!} \varepsilon^{m} (\eta(\dd)\kappa_0 r^{\dd})^{m-1} \\
   & \leq n\varepsilon (\constB \varepsilon n r^{\dd})^{m-1}.
\end{align*}
\end{proof}

\begin{lemma}
  \label{lem:probomegabar}
  Assume that assumptions {\bf (A1)} and {\bf (A6)} hold. Define $N_i=|\mathbb{X}_n\cap S_{i}|,i\in\intervalb{0}{M}$ and for $0<\eta\leq \eta_0$ let
  \begin{equation*}
    \Omega_\eta=\bigcap_{i=1}^M\left\{(1-\eta)(1-\varepsilon)\gamma_in<N_i<(1+\eta)(1-\varepsilon)\gamma_in\right\}.
  \end{equation*}
  We have
  \begin{itemize}
    \item[(i)] $\prob(\overline{\Omega_\eta})\leq 2M\exp(-\psi(\eta)(1-\varepsilon)\gbar n)$;
    \item[(ii)] $N_0<\frac{\gbar}{\gamma_*}\min_{i\in\interval{M}}N_i<\min_{i\in\interval{M}}N_i$ under $\Omega_\eta$.
  \end{itemize}
\end{lemma}

\begin{proof}[\textbf{\upshape Proof:}] Since $N_i\sim B(n,(1-\varepsilon)\gamma_i)$, (i) is a direct consequence of~\citet[page 440]{showell86}. For (ii), observe that $(1-\varepsilon)(1-\eta_0)=1/(1+\gbar)$. Since $0<\eta\leq\eta_0$, it follows that
  \begin{equation*}
    1-(1-\varepsilon)(1-\eta)\leq (1-\varepsilon)(1-\eta)\gbar.
  \end{equation*}
  Thus, under $\Omega_\eta$,
  \begin{align*}
  N_0
   & \leq n-\sum_{i=1}^MN_i\leq n\left(1-(1-\varepsilon)(1-\eta)\sum_{i=1}^M\gamma_i\right)                              \\
   & \leq n(1-(1-\varepsilon)(1-\eta)) \leq n(1-\varepsilon)(1-\eta)\gbar                                                \\
   & \leq \frac{\gbar}{\gamma_*}\,n(1-\varepsilon)(1-\eta)\gamma_i<\frac{\gbar}{\gamma_*}\,N_i\quad\forall i=1,\dotsc,M.
\end{align*}
\end{proof}

\begin{lemma}
  \label{lem:borneetagamma}
  Assume that assumption {\bf (A6)} holds. For each $\eta\leq\min(\eta_0,\eta_1)$ we have
  \begin{equation*}
    \frac{1+\eta}{1-\eta}\frac{\gamma^*}{\gamma_*}+\frac{\gbar}{\gamma_*}\leq 2.
  \end{equation*}
\end{lemma}
\begin{proof}[\textbf{\upshape Proof:}]
  Let $\eta\leq\min(\eta_0,\eta_1)$, then
  \begin{align*}
  \frac{1+\eta}{1-\eta}\frac{\gamma^*}{\gamma_*}+\frac{\gbar}{\gamma_*}
   & =\frac{1-\eta+2\eta}{1-\eta}\frac{\gamma^*}{\gamma_*}+\frac{\gbar}{\gamma_*}                                                     \\
   & =\frac{\gamma*+\gbar}{\gamma_*}+\frac{1}{2}\,\frac{4\eta}{1-\eta}\,\frac{\gamma^*}{\gamma_*}                                     \\
   & \leq \frac{\gamma^*/2+\gamma_*}{\gamma_*}+\frac{1}{2}\left(\frac{\gamma_*}{\gamma^*}-\frac{1}{2}\right)\frac{\gamma^*}{\gamma_*} \\
   & =\frac{3}{2}+\frac{\gamma^*}{4\gamma_*}\leq 2.
\end{align*}
\end{proof}

\begin{lemma}
  \label{lem:techeta0eta1}
  Assume {\bf (A6)} holds. For $r>0$, let
  \begin{equation*}
    \mathcal E(r)=\left\{\exists\pi\in\Pi_M\,\forall i=1,\dotsc,M\ {\mathbb{X}_n}\cap S_{i}\subseteq\cX_{\pi(i)}(r)\right\}.
  \end{equation*}
  Let $\eta\leq \min(\eta_0,\eta_1)$, then under $\mathcal E(r)\cap \Omega_\eta$ we have
  \begin{enumerate}
    \item $\widehat r_n\geq r$ almost surely;
    \item There exists $\pi\in\Pi_M$ such that, $\forall i=1,\dotsc,M$ $\cX_i(r)\subseteq\cX_{\pi(i)}(\widehat r_n)$.
  \end{enumerate}
\end{lemma}

\begin{proof}[\textbf{\upshape Proof:}]
  Let $\eta\leq \min(\eta_0,\eta_1)$ and assume that $\Omega_\eta$ is true. We first prove that $\widehat r_n\geq r$ with a reductio ad absurdum. Assume that $\widehat r_n< r$. Observe that
  \begin{equation*}
    \left|\mathcal Y_M(\widehat r_n)\right|>\left|\mathcal Y_M(r)\right|,
  \end{equation*}
  by definition of $\widehat r_n$. 
  It follows that
  \begin{equation*}
    \left|\mathcal Y_1(\widehat r_n)\right|\geq \dotsc\geq \left|\mathcal Y_M(\widehat r_n)\right|>\left|\mathcal Y_M(r)\right|.
  \end{equation*}
  Since $\widehat r_n< r$, we deduce that one of the $\mathcal Y_i(r),i=1,\dotsc,M-1$  contains observations of at least two clusters among $\mathcal Y_i(\widehat r_n),i=1,\dotsc,M$. It implies that
  \begin{equation}
    \label{eq:rapportsizeYabsurde}
    \left|\mathcal Y_{1}(r)\right|\geq 2 \left|\mathcal Y_M(\widehat r_n)\right|>2\left|\mathcal Y_M(r)\right|.
  \end{equation}
  Moreover, under $\mathcal E(r)$ we have $N_{(1)}\leq \left|\mathcal Y_1( r)\right|\leq N_{(1)}+N_0$ where $N_i=|\mathbb{X}_n\cap S_{i}|$ and $N_{(i)},i=1,\dotsc,M$ are such that
  \begin{equation*}
    N_{(1)}\geq \dotsc\geq N_{(M)}.
  \end{equation*}
  Thus, under $\mathcal E(r)\cap \Omega_\eta$, we have from Lemma~\ref{lem:probomegabar}
  \begin{equation*}
    \left|\mathcal Y_{1}(r)\right|\leq N_{(1)}+N_0\leq N_{(1)}+\frac{\gbar}{\gamma_*}N_{(M)}.
  \end{equation*}
  Since $\left|\mathcal Y_{(M)}(r)\right|\geq N_{M}$, we obtain from Lemma~\ref{lem:borneetagamma}
  \begin{equation*}
    \frac{\left|\mathcal Y_{1}(r)\right|}{\left|\mathcal Y_{M}(r)\right|}\leq \frac{N_{(1)}}{N_{(M)}}+\frac{\gbar}{\gamma_*}\leq \frac{1+\eta}{1-\eta}\frac{\gamma^*}{\gamma_*}+\frac{\gbar}{\gamma_*}\leq 2,
  \end{equation*}
  which is a contradiction with~\eqref{eq:rapportsizeYabsurde}. We deduce that $\widehat r_n\geq r$ almost surely.

  For the second point, observe that since $\widehat r_n\geq r$, each $\cX_i(\widehat r_n),i=1,\dotsc,M$ may be written as a union of clusters in
  \begin{equation*}
    \cX_1(r),\dotsc,\cX_M(r),\mathcal Y_{M+1}(r),\dotsc,\mathcal Y_{M(r)}(r).
  \end{equation*}
  Moreover, for each $i\in\{1,\dotsc,M\}$ there exists a unique $j\in\{1,\dotsc,M\}$ and a subset $\mathcal T(\widehat r_n)$ of $\{M+1,\dotsc,M(r)\}$ such that
  \begin{equation}
    \label{eq:inclu_r-rcha}
    \cX_i(\widehat r_n)=\cX_j(r)+\bigcup_{\ell\in\mathcal T(\widehat r_n)} \mathcal Y_\ell(r).
  \end{equation}
  Indeed, if there exists $i\in\{1,\dotsc,M\}$ and $1\leq j\neq j'\leq M$ such that
  \begin{equation*}
    \cX_j(r)\cup \cX_{j'}(r)\subseteq \cX_i(\widehat r_n),
  \end{equation*}
  then $\cX_M(\widehat r_n)$ may be written as a union of clusters in
  \begin{equation*}
    \{\mathcal Y_{M+1}(r),\dotsc,\mathcal Y_{M(r)}(r)\},
  \end{equation*}
  and thus $\left|\cX_M(\widehat r_n)\right|\leq N_0$. This is not possible since, by definition of $\widehat r_n$ and by Lemma~\ref{lem:probomegabar}, we have
  \begin{equation*}
    \left|\cX_M(\widehat r_n)\right|\geq \left|\cX_M(r)\right|\geq N_{(M)}> N_0.
  \end{equation*}
  We deduce that~\eqref{eq:inclu_r-rcha}  is true. Therefore, there exists $\pi\in\Pi_M$ such that, $\forall i=1,\dotsc,M$ $\cX_i(r)\subseteq\cX_{\pi(i)}(\widehat r_n)$.
\end{proof}

\subsection{Proof of Theorems}
\begin{proof}[\textbf{\upshape Proof of Theorem~\ref{theo:clustriskhatrn}:}]
  For $r>0$, let
  \begin{equation}
    \label{eq:defmathcal E}
    \mathcal E(r)=\left\{\exists\pi\in\Pi_M,\,\forall i=1,\dotsc,M,\ {\mathbb{X}_n} \cap S_{i}\subseteq\cX_{\pi(i)}(r)\right\}
  \end{equation}
  Let $\eta\leq\min(\eta_0,\eta_1)$. Observe that
  \begin{align*}
    1-\mathcal R_n(\cX(\widehat r_n)) & = \prob\big(\mathcal E(\widehat r_n)\big)                              \\
                                      & \geq \prob\big(\mathcal E(\widehat r_n),\mathcal E(r),\Omega_\eta\big) \\
                                      & =\prob\big(\mathcal E(r),\Omega_\eta\big)
  \end{align*}
  where last line comes from Lemma~\ref{lem:techeta0eta1}. We deduce that
  \begin{align*}
    \mathcal R_n(\cX(\widehat r_n)) & \leq 1-\prob(\Omega_\eta,\mathcal E(r))                   \\
                                    & \leq 1-\prob(\mathcal E(r))+\prob(\overline{\Omega_\eta}) \\
                                    & \leq \mathcal R_n(\cX(r))+\prob(\overline{\Omega_\eta})
  \end{align*}
  and the result follows from Lemma~\ref{lem:probomegabar}.
\end{proof}

\begin{proof}[\textbf{\upshape Proof of Theorem~\ref{thm:cluster-identification}:}]
  First observe that for $r>0$,
  \begin{align}
    \label{eq:minoran1moinsR}
    1-\mathcal R_n(\cX(r)) 
    & =\prob\left(\exists\pi\in\Pi_M\,\forall i=1,\dotsc,M\ {\mathbb{X}_n}\cap S_{i}\subseteq\cX_{\pi(i)}(r)\right) \nonumber       \\
    & \geq \prob\left(\exists\pi\in\Pi_M\,\forall i=1,\dotsc,M\ {\mathbb{X}_n}\cap S_{i}\subseteq\cX_{\pi(i)}(r),\Omega_\eta\right)
  \end{align}
  where $\Omega_\eta$ is the event defined in Lemma~\ref{lem:probomegabar}. Since, under $\Omega_\eta$, $N_0<\min_{i\in\interval{M}}N_i$ the event in~\eqref{eq:minoran1moinsR} equals
  \begin{equation*}
    \left\{
    \begin{array}{l}
      \forall i=1,\dotsc,M,\,{\mathbb{X}_n}\cap S_{i}\text{ are }r\text{-connected}                                                   \\
      \forall i\neq j,\text{ there is no }r\text{-connected path between } \mathbb{X}_n\cap S_{i}\text{ and }\mathbb{X}_n\cap S_{\!j} \\
      \Omega_\eta,
    \end{array}\right.
  \end{equation*}
  which contains (since $0<r<\delta$)
  \begin{equation*}
    \left\{
    \begin{array}{l}
      \forall i=1,\dotsc,M,\,{\mathbb{X}_n}\cap S_{i}\text{ are }r\text{-connected}                           \\
      \text{there is no $r$-connected path in $S_{0}$ with at least $\lfloor \delta/r\rfloor+1$ observations} \\
      \Omega_\eta.
    \end{array}\right.
  \end{equation*}
  We deduce from Lemmas~\ref{lem:within-cluster-connectivity} and~\ref{lem:between-cluster-connectivity} that
  \begin{align*}
    \label{eq:risqueRn3termes}
    \mathcal R_n(\cX(r))\leq & \sum_{i=1}^M\psi_{n,i}(r)+\varphi_n\left(\left\lfloor \frac{\delta}{r}\right\rfloor+1,r\right)+\prob(\overline{\Omega_\eta})                                  \\
    \leq                     & \Lambda r^{-d} \exp(-\constA nr^{d})+n\varepsilon (\constB \varepsilon n r^{\dd})^{\left\lfloor \frac{\delta}{r}\right\rfloor}+\prob(\overline{\Omega_\eta}),
  \end{align*}
  where $\Lambda=\sum_{i=1}^M\Lambda_i$. Result follows from Lemma~\ref{lem:probomegabar}.
\end{proof}

\bibliography{kr}

\end{document}